\documentclass[preprint,11pt]{elsarticle}

\usepackage{amsmath,amsthm,amsfonts,amssymb}
\usepackage{graphics,graphicx}
\usepackage{hyperref}
\usepackage{pifont}
\usepackage{natbib}
\usepackage{fleqn}

\usepackage{txfonts}
\usepackage{geometry}
\newtheorem{theorem}{Theorem}[section]
\newtheorem{lemma}{Lemma}[section]

\theoremstyle{remark}
\newtheorem{remark}{Remark}[section]

\newcommand{\e}{^\varepsilon}
\newcommand{\eps}{{\varepsilon}}
\newcommand{\ds}{\displaystyle}
\newcommand{\I}{\mathcal{J}\e}
\newcommand{\<}{\langle}
\renewcommand{\>}{\rangle}
\renewcommand{\a}{\alpha}
\renewcommand{\b}{\beta}

\newcommand{\cupl}{\bigcup\limits}
\newcommand{\supp}{\mathrm{supp}}
\newcommand{\suml}{\sum\limits}
\newcommand{\intl}{\int\limits}
\newcommand{\liml}{\lim\limits}
\newcommand{\maxl}{\max\limits}

\renewcommand{\phi}{\varphi}

\newcommand{\x}{\textbf{x}}

\renewcommand{\d}{\hspace{1pt}\mathrm{d}}

\numberwithin{equation}{section}

\begin{document}

\begin{frontmatter}

\title{Spectrum of a singularly perturbed periodic thin waveguide}

\author[a1]{Giuseppe Cardone}
\ead{giuseppe.cardone@unisannio.it}
\address[a1]{Department of Engineering, University of Sannio, Corso Garibaldi 107, 82100 Benevento, Italy}
\cortext[cor1]{Corresponding author}

\author[a2]{Andrii Khrabustovskyi\corref{cor1}}
\ead{andrii.khrabustovskyi@kit.edu}
\address[a2]{Institute of Analysis, Karlsruhe Institute of Technology, Englerstr. 2, 76131 Karlsruhe, Germany}

\journal{}

\begin{abstract}
We consider a family $\{\Omega^\varepsilon\}_{\varepsilon>0}$ of periodic domains in $\mathbb{R}^2$ with waveguide geometry and analyse  spectral properties of the Neumann Laplacian $-\Delta_{\Omega^\varepsilon}$ on $\Omega^\varepsilon$.
The waveguide $\Omega^\varepsilon$ is a union of a thin straight strip of the width $\varepsilon$ and a family of small protuberances with the so-called ``room-and-passage'' geometry. The protuberances are attached periodically, with a period $\varepsilon$, along the strip upper boundary. For $\varepsilon\to 0$ we prove a (kind of) resolvent convergence of $-\Delta_{\Omega^\varepsilon}$ to a certain ordinary differential operator. Also we demonstrate Hausdorff convergence of the spectrum.
In particular, we conclude that if the sizes of ``passages'' are appropriately scaled the first spectral gap of $-\Delta_{\Omega^\varepsilon}$ is determined exclusively by geometric properties of the protuberances. The proofs are carried out using methods of homogenization theory.
\end{abstract}

 
\end{frontmatter}

\section{Introduction}

In the paper we study the limiting behaviour as $\eps\to 0$ of the
Neumann Laplacian on a thin periodic domain $\Omega\e\subset\mathbb{R}^2$ with waveguide geometry -- see Figure \ref{fig1}. 
The domain $\Omega\e$ is obtained from the straight unbounded strip $\Pi\e= \mathbb{R}\times (0,\eps)$ by attaching
an array of small identical protuberances (counted by the parameter $j\in \mathbb{Z}$). Each protuberance consists of 
two parts (below $\simeq$ means that domains coincide up to a translation):  
\begin{itemize}
\item \textit{the ``room''} $B_j\e\simeq\eps B$, where $B\subset\mathbb{R}^2$ is a fixed domain,

\item  \textit{the ``passage''} $T_j\e\simeq (0,d\e)\times [0,h\e]$ connecting the ``room'' $B_j\e$ with the strip $\Pi\e$. Here $h\e\to 0$, $d\e=o(\eps)$ as $\eps\to 0$.
\end{itemize}
The protuberances $T_j\e\cup B_j\e$, $j\in\mathbb{Z}$ are attached periodically, with a period  $\eps$, along the strip upper boundary. 

\begin{figure}[h]
\begin{center}
\begin{picture}(320,80)
\scalebox{0.4}{\includegraphics{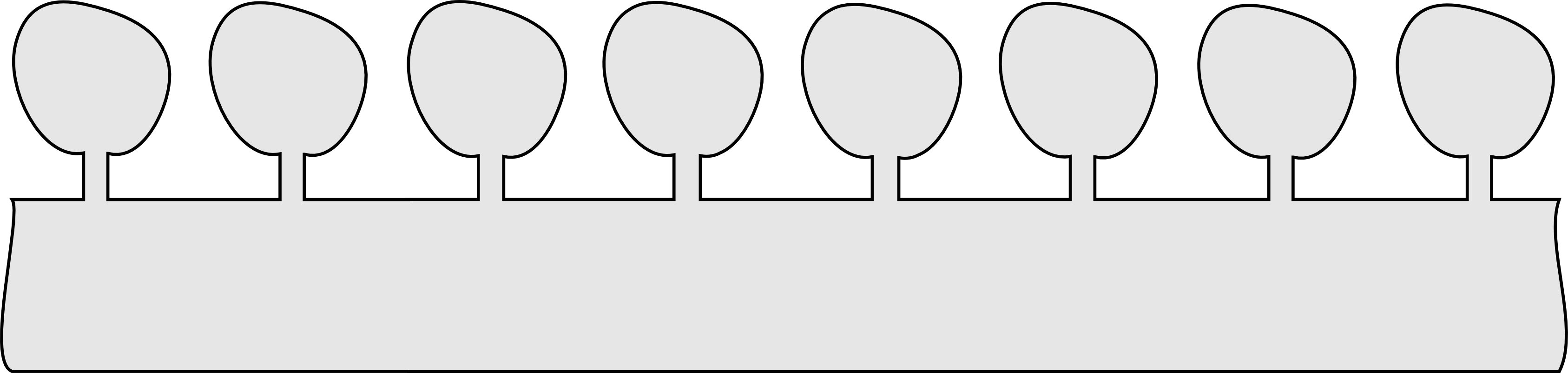}}
\put(-291,44){$T_j\e$}\put(-292,44){\vector(-4,-1){13}}
\put(-315,59){$B_j\e$}
\put(-300,15){$\Pi\e$}

\put(-8,21){\vector(0,1){14}}
\put(-8,21){\vector(0,-1){20}}
\put(-6,20){$_\eps$}

\put(1,59){\vector(0,1){18}}
\put(1,59){\vector(0,-1){15}}
\put(2,57){$_{\mathcal{O}(\eps)}$}

\put(-34,29){\vector(1,0){16}}
\put(-34,29){\vector(-1,0){26}}
\put(-41,32){$_\eps$}

\put(-8,34){\vector(0,1){10}}
\put(-8,40){\vector(0,-1){6}} 
\put(-6,40){$_{h\e}$}

\put(-29,39){\vector(1,0){8}}
\put(-13,39){\vector(-1,0){3}}
\put(-35,40){$_{d\e}$}

\end{picture}\caption{Domain $\Omega\e$}\label{fig1}
\end{center}
\end{figure}

Peculiar properties of Neumann spectral problems on domains 
perturbed by attaching small ``room-and-passage'' protuberances
were observed for the first time by R.~Courant and D.~Hilbert   \citep[Volume~I, Chapter~VI, \S~2.6]{CH53}. Below we sketch the example considered in \citep{CH53}.
Let us perturb a bounded connected domain $\Omega$ 
to a domain $\Omega\e$ by attaching a single ``room-and-passage'' protuberance.
The domain $\Omega\e$ differs from $\Omega$ only in a ball of the radius tending to zero as $\eps\to 0$. One can easily show (using, e.g., \cite[Theorem~1.5]{RT75}) that for each $k\in\mathbb{N}$ the $k$-th eigenvalue of \textit{the Dirichlet Laplacian} on $\Omega\e$ converges to the $k$-th eigenvalue of the Dirichlet Laplacian on $\Omega$. In contrast, for the Neumann Laplacians (we denote them $-\Delta_{\Omega\e}$ and $-\Delta_{\Omega}$) the continuity of eigenvalues does not hold in general: the first eigenvalues $\lambda_1$ and $\lambda_1\e$ of  $-\Delta_{\Omega }$ and $-\Delta_{\Omega\e}$ are zero, the second eigenvalue $\lambda_2$ of $-\Delta_{\Omega}$ is strictly positive, while the second eigenvalue $\lambda_2\e$ of $-\Delta_{\Omega\e}$ tends to zero as $\eps\to 0$ provided $h\e=\eps$, $d\e=\eps^\a$, $\a>3$.  

Later the aforementioned example was studied in \citep{AHH} for more general geometry of ``rooms'' and ``passages''. The authors also inspected the case of 
\textit{finitely} many attached ``room-and-passage'' domains proving that  
$\liml_{\eps\to 0}\lambda_k\e=0$ as $k=2,\dots,r+1$ and $\liml_{\eps\to 0}\lambda_k\e=\lambda_{k-r}$ as $k\geq r+2$, where $r\in\mathbb{N}$ is the number of attached domains.

The case, when  the number of attached ``room-and-passage''  protuberances tends to infinity as $\eps\to 0$, was studied in our previous paper \citep{CK15} (still with a bounded $\Omega$). We considered the operator
\begin{gather*}
\mathcal{H}\e=-(\rho\e)^{-1}\Delta_{\Omega\e},
\end{gather*}
where $-\Delta_{\Omega\e}$ is the Neumann Laplacian on $\Omega\e$, the weight $\rho\e$ is equal to $1$ everywhere except  the union of the ``rooms'', where it is equal to the constant $\varrho\e>0$ satisfying $\liml_{\eps\to 0}{\varrho\e \eps}<\infty$.
It was proved that the spectrum $\sigma(\mathcal{H}\e)$ of the operator $\mathcal{H}\e$ converges in the Hausdorff sense as $\varepsilon\to 0$   to the set $\sigma(\mathcal{H})\cup\{q\}$, where
$q=\liml_{\eps\to 0}{d\e\over h\e\varrho\e \eps^2 |B|}\in [0,\infty]$  (by $|\cdot|$ we denote the Lebesgue measure of a domain) 
and $\mathcal{H}$ is the operator associated with
the following spectral problem:
\begin{gather*}
 \ds -\Delta u=\lambda u \text{ in }\Omega,\quad 
{\partial u\over\partial  n}=\mathcal{V}(\lambda) u\text{ on }\Gamma,\quad
{\partial u\over\partial  n}=0 \text{ on }\partial\Omega\setminus\Gamma,
\end{gather*}
where $\Gamma$ is  the perturbed part of $\partial\Omega$, $n$ is the outward-pointing unit normal to $\partial\Omega$.
If $\liml_{\eps\to 0}{\varrho\e \eps}= 0$ one has $\mathcal{V}(\lambda)\equiv 0$, otherwise $\mathcal{V}(\lambda)$ is either linear function ($q=\infty$)  or rational 
 function ($q<\infty$) with a pole at $q$, which is also a point of accumulation of eigenvalues of $\mathcal{H}$. 
 
Note, that in the case $\varrho\e=1$ (i.e., $\mathcal{H}\e$ is simply the Neumann Laplacian on $\Omega\e$) one has $\mathcal{V}(\lambda)\equiv 0$, i.e. $\mathcal{H}$ is the Neumann Laplacian on $\Omega$.
The case $\varrho\e=1$, $q=0$ was also studied  in \citep[Chapter XII]{SP80}.

The results of \citep{CK15} were extended  in \citep{CK16}\footnote{In fact, in \citep{CK16} we deal with the most interesting case $q>0$, $\liml_{\eps\to 0}{\varrho\e \eps}>0$ only. The analysis for the rest cases can be carried out a similar way.} to $\Omega$, which is an unbounded straight strip of the \textit{fixed} width $L>0$. In this case $\Gamma$ is its upper (or lower) boundary. It turns out that the form of the limit problem remains the same as in the case of a bounded domain, but the structure of its spectrum is essentially different: it is either the whole positive semi-axis or the set $[0,\infty)\setminus (q,\widehat q)$ (this case occur if $\liml_{\eps\to 0}{\varrho\e \eps}>0$ and $0<q<\left(\pi\over 2L\right)^2$). The number $\widehat{q}\in (q,\infty)$ is a solutions to some transcendental equation involving $q$ and $L$. In the last instance we are able to make the following useful conclusion: the spectrum of $\mathcal{H}\e$ has a gap provided $\eps$ is small enough, the edges of this gap converge to $q$ and $\widehat q$. 
\medskip

In the current paper we study the asymptotic behaviour of the ``pure'' Neumann Laplacian (i.e., $\varrho\e=1$), but now, in contrast to \citep{CK16}, the ``basic'' strip $\Omega=\Pi\e$ also depends on $\eps$. Since $\Omega\e$ shrinks to $\mathbb{R}$ as $\eps\to 0$ it is natural to expect that $-\Delta_{\Omega\e}$  converges (in suitable sense) to some ordinary differential operator on the line.
It turns out that the form of this operator depends on  
$q=\liml_{\eps\to 0}{d\e\over h\e\eps^2|B|}\in [0,\infty]$.  

Boundary value and spectral problems  on \textit{thin domains with oscillating boundary} were studied in a lot works -- see, e.g., \cite{AP11,HR92,PS13,PS15} and references therein. In these papers the authors deal with thin domains, whose boundary (or its part) has the form of a graphic of some smooth 
oscillating function (for example, $\Omega\e=\left\{\mathbf{x}\in\mathbb{R}^2:\   -\eps<x_2< \eps \varphi\e(x_1)\right\},$
where $\eps>0$ is again a small parameter, $\varphi\e(x)=\varphi(x,x\eps^{-\a})$, $\a>0$, $\varphi(x,y)$ is a smooth positive function, periodic with respect to $y$).
More general geometries were treated in \cite{AP13}, where thin strip is perturbed 
by attaching small protuberances, $\eps^\a$-periodically along it boundary;
the protuberances are obtained from a fixed bounded domain by $\eps^\a$-rescaling in $x_1$ direction and $\eps$-rescaling in $x_2$ direction.
In \cite{MP10,MP12}, besides an oscillating external boundary, additional internal holes  are allowed. 
\medskip

At first, we study the resolvent equation
\begin{gather}\label{resolvent-eq}
-\Delta_{\Omega\e}u\e+\mu u\e=f\e,\ \mu>0.
\end{gather}
We prove (see Theorems \ref{th1}-\ref{th1+})  that under some natural assumptions on
$f\e$ the solution $u\e$ to the problem \eqref{resolvent-eq} converges in a suitable sense to the solution
of the following problem on the line:
\begin{gather} \label{resolvent}
-u''(x)+\mathcal{V}(\mu)u(x)=\mathcal{F}(\mu,x).
\end{gather}
The functions 
$\mathcal{V}(\mu)$, $\mathcal{F}(\mu,x)$  are either linear ($q=\infty$)  or rational ($q<\infty$)
 functions  of $\mu$. In the later case they both have one pole -- at the point $q$.

Problem \eqref{resolvent} can be associated with a resolvent equation
for some self-adjoint operator $\mathcal{H}$ acting in $[L_2(\mathbb{R})]^2$. The spectrum of this operator has the form
\begin{gather}\label{sigmaH}
\sigma(\mathcal{H})=[0,\infty)\setminus \left(q,q+q|B|\right)
\end{gather}
provided $q>0$, otherwise $\sigma(\mathcal{H})=[0,\infty)$.

Our second result concerns the spectral convergence in the most interesting case $q<\infty$. 
Periodicity of $\Omega\e$ leads to the band structure of $\sigma(-\Delta_{\Omega\e})$, 
i.e. $\sigma(-\Delta_{\Omega\e})$  is a locally finite union of compact intervals called \textit{bands}. In general they may overlap, otherwise we have a \textit{gap} in the spectrum -- a  bounded open interval having an empty intersection with the spectrum  but with ends belonging to it.

We prove (see Theorem \ref{th2}) that the spectrum of $-\Delta_{\Omega\e}$ converges as $\eps\to 0$ to the spectrum of $\mathcal{H}$ in the Hausdorff sense. This means that 
$\sigma(-\Delta_{\Omega\e})$ has a gap provided $\eps$ is small enough;
when $\eps\to 0$ this gap converges to the interval $(q,q+q|B|)$.
Moreover, we show (see Lemma \ref{lm1}) that other gaps (if any) ``escape'' from any finite interval when $\eps\to 0$.

Theorems \ref{th1}-\ref{th2} remain valid if $\Omega$ is a bounded strip, cf. Remark \ref{rm-bounded} below.  \medskip

Note, that using the same ideas one can also construct a waveguide with several gaps. Namely, if we attach to $\Pi\e$\quad $m\in\mathbb{N}$ 
different families of ``room-and-passage'' domains we will arrive at the same 
limit problem \eqref{resolvent}, but the functions $\mathcal{V}(\mu)$, $\mathcal{F}(\mu,x)$ will have $m$ poles. The corresponding operator $\mathcal{H}$ will act in $  [L_2(\mathbb{R})]^{m+1}$ and have $m$ gaps.
The proof relies on the same methods as in the case $m=1$, but is more cumbersome. 
\medskip

The possibility to open up gaps in the spectrum of periodic differential operators is important from the point of view of applications, in particular, to the so-called photonic crystals --  periodic nanostructures that have been attracting much attention in recent years. The characteristic property of photonic crystals is that the light waves at certain optical frequencies fail to propagate in them, which is caused by gaps in the spectrum of the Maxwell operator or related scalar operators. 
Pioneer mathematical results justifying the opening of spectral gaps for some $2D$ dielectric media were obtained in \cite{FK96}. We refer to the overview \cite{K01} and the book \citep{DLPSW11} concerning mathematical problems arising in this field.

As we already mentioned, in general, the presence of gaps is not guaranteed -- for instance, if $\Omega$ is a  straight unbounded strip then  the spectrum of the Laplace operator is a ray $[\Lambda,\infty)$, where  $\Lambda=0$ for the Neumann Laplacian  and $\Lambda>0$ for the Dirichlet Laplacian. There exist several approaches how to construct a periodic waveguide-like domain  with non-void gaps in the spectrum of the Laplace operator on this domain subject to Neumann or Dirichlet boundary conditions. 
The simplest way is to consider the waveguide consisting of an array of identical compact domains connected by thin passages or windows  -- see, e.g.,  \citep{BN15,B15,P10}. In this case the spectrum typically has small bands 
separated by relatively large gaps.  
The opposite picture (i.e., large bands versus small gaps) occurs under ``small'' perturbations of straight waveguides (see \citep{BNR13,CNP10,N10_1,N12,N15}) --- either by a periodic nucleation of small  holes or by a gentle periodic bending of the boundary. 
The waveguide $\Omega\e$ constructed in the current paper falls into an intermediate case -- the length of the first band is comparable with the length of the first gap. Moreover, in contrast to \citep{CK16}, both edges of this gap depends on geometric properties of the waveguide in a very simple fashion.

Thin periodic waveguides of constant width were treated in  \cite{Y98}. It was proved that
the Dirichlet Laplacian on such a waveguide always has at least one gap provided the signed  curvature of  the  boundary  curve  is smooth and non-constant and the waveguide is thin enough. The opening of spectral gaps for the Dirichlet Laplacian on the waveguide of the form $\left\{(x_1,x_2)\in\mathbb{R}^2:\ 0<x_2<\eps h(x_1)\right\}$, where $h(x)$ is a positive periodic function, was established in \cite{FS08} under a suitable assumptions on $h$.
The waveguide consisting of two parallel strips coupled through a period family of thin windows 
was studied in \cite{BP13_1}.
Finally, we mention the papers \citep{BP13_2,CMN09,KP16,NRT10,N10_2} where 
the same type problems were addressed for more general elliptic selfadjoint operators, and
\cite{BRT15}, where the Steklov spectral problem was considered.\medskip

The paper is organized as follows. In Section \ref{sec2} we set up the problem and formulate the main results. 
We prove  resolvent convergence of 
$-\Delta_{\Omega\e}$ in Sections \ref{sec3} ($q<\infty$) and \ref{sec4} ($q=\infty$). In Section~\ref{sec5} we prove  Hausdorff convergence of the spectrum. Finally, in Section \ref{sec6} we show that $-\Delta_{\Omega\e}$ has at most one gap on finite intervals provided $\eps$ is small enough.

\section{Setting of the problem and main results \label{sec2}}

Let $\eps>0$ be a small parameter, and $d\e,h\e$ be positive numbers  depending on $\eps$ and satisfying
\begin{gather}
\label{ass_conc}
\liml_{\eps\to 0}{d\e\over \eps}=0,
\\\label{ass2}
\liml_{\eps\to 0}h\e = 0,
\\
\label{ass1}
\liml_{\eps\to 0}{\eps^2 \ln d\e} = 0.
\end{gather}
Condition \eqref{ass1} is rather technical, one needs it to have a better
control on the behaviour of functions from $H^1(\Omega\e)$ near 
the bottom and the top of the passages (see Lemma \ref{lm2} below).

Hereinafter by $\x=(x_1,x_2)$ we denote the points in $\mathbb{R}^2$,
by $x$ we denote the points in $\mathbb{R}$.
Further, we introduce the following sets (below $j\in\mathbb{Z}$):
\begin{itemize}
\item $\Pi\e=\left\{\x \in \mathbb{R}^2:\ -\eps< x_2< 0\right\}$,

\item $T_j\e=\left\{\x \in \mathbb{R}^2:\ |x_1-x_j\e|<\ds{d\e\over 2},\ 0\leq x_2\leq h\e\right\}$, where $x_j\e=\eps(j+1/2)$,

\item $B_j\e=\left\{\mathbf{x}\in \mathbb{R}^2:\ \mathbf{x}-\widetilde{\textbf{x}}_j\e\in \eps B\right\}$, where $ \widetilde{\textbf{x}}_j\e=(x_j\e,h\e)\in\mathbb{R}^2$, 
$B\subset\mathbb{R}^{2}$ is an open bounded domain with Lipschitz boundary such that
\begin{gather}\label{assB1}
B\subset \left\{\x \in\mathbb{R}^2:\ |x_1|<{1\over 2},\ 
x_2>0\right\},\\\label{assB2}
\exists R\in \left(0,{1\over 2}\right):\ \left\{\x \in\mathbb{R}^2:\ |x_1|<R,\ x_2=0\right\}\subset \partial B.
\end{gather}
\end{itemize}
By virtue of the condition \eqref{assB1}  the neighbouring ``rooms'' $B_j\e$ are pairwise disjoint.  
Condition \eqref{assB2} together with \eqref{ass_conc}  imply the correct gluing of the $j$-th ``room'' and the $j$-th ``passage'', namely the upper face of $T_j\e$ is contained in $\partial B_j\e$.

Finally, we define the waveguide $\Omega\e$ as a union of the straight strip $\Pi\e$ and $\eps$-periodically attached ``room-and-passage''  protuberances $B_j\e\cup T_j\e$ (see Figure \ref{fig1}):
$$\Omega\e=\Pi\e\cup\left(\cupl_{j\in\mathbb{Z}}B_j\e\cup T_j\e\right).$$

We denote by $\mathcal{H}\e$ the Neumann Laplacian in $L_2(\Omega\e)$ -- the
self-adjoint and positive operator associated with
the sesquilinear form $\mathfrak{h}\e$,
$$\mathfrak{h}\e [u,v]=\intl_{\Omega\e}\nabla u(\x)\cdot\overline{\nabla v(\x)}\d \mathbf{x},\quad \mathrm{dom}(\mathfrak{h}\e)=H^1(\Omega\e)$$
(i.e., $(\mathcal{H}\e u,v)_{L_2(\Omega\e)}=\mathfrak{h}\e [u,v]$, $\forall u\in \mathrm{dom}(\mathcal{H}\e)$, $\forall v\in \mathrm{dom}(\mathfrak{h}\e)$).
\smallskip

Our first goal is to describe the behaviour of the resolvent $(\mathcal{H}\e+\mu I)^{-1}$, $\mu>0$ as $\eps\to 0$
under the assumption that the following limit $q$, either finite or infinite, exists:
\begin{gather}
\label{ass3}
\liml_{\eps\to 0}{d\e\over h\e\eps^2|B|}=q\in [0,\infty].
\end{gather}
Note, that if $q<\infty$ then \eqref{ass_conc} follows automatically from \eqref{ass2}, \eqref{ass3}.\smallskip

Before to formulate the result we need to introduce auxiliary operators.

We define the operator $J_1\e: L_2(\Pi\e)\to L_2(\mathbb{R})$
by the formula
\begin{gather}\label{j1}
(J_1\e u)(x)={1\over\sqrt{\eps}} \int_{-\eps}^0 u(\x) \d x_2,\text{ where }\x=(x,x_2).
\end{gather}
Also we introduce the operator 
$J_2\e : L_2(\cupl_{j\in\mathbb{Z}} B_j\e)\to L_2(\mathbb{R})$ by
\begin{gather}\label{j2}
(J_2\e u)(x)=\suml_{j\in\mathbb{Z}}\left({1\over \sqrt{\eps|B_j\e|}}\intl_{B_j\e} u(\mathbf{x}) \d \mathbf{x}\right)\chi_j(x),
\end{gather}
where $\chi_j(x)$ is the indicator function of the interval $\left(x_j\e-{\eps\over 2},x_j\e+{\eps\over 2}\right)$.

Using the Cauchy-Schwarz inequality we get 
\begin{gather}\label{iso0}
\forall f\in L_2(\Omega\e):\quad \|J_1\e f\|_{L_2(\mathbb{R})}\leq \|f\|_{L_2(\Pi\e)},\quad 
\|J_2\e f\|_{L_2(\mathbb{R})}\leq \|f\|_{L_2(\cupl_{j\in\mathbb{Z}}B_j\e)}.
\end{gather}
Moreover, it is easy to show that $J\e_1 u\in H^1(\mathbb{R})$ provided $u\in H^1(\Omega\e)$
and the following Poincar\'{e}-type estimates are valid:
\begin{gather}\label{iso1}
\|u\|^2_{L_2(\Pi\e)}\leq \|J_1\e u\|^2_{L_2(\mathbb{R})}+ C\eps^2\|\nabla u\|^2_{L_2(\Pi\e)},\\\label{iso2}
\|u\|^2_{L_2(\cupl_j{B_j\e})}\leq\|J_2\e u\|^2_{L_2(\mathbb{R})}+ C\eps^2\|\nabla u\|^2_{L_2(\cupl_j B_j\e)}.
\end{gather}
Hereinafter by $C,C_1,C_2,\dots$ we denote generic constants which do not depend on $\eps$. 
Inequalities \eqref{iso0}-\eqref{iso2} mean that $J_1\e,J_2\e$ are ``almost'' isometries (as $\eps \ll  1$). \medskip

We are now in position to formulate the first result; it deals with the most interesting case $q<\infty$. Below, as usual,\, $\rightharpoonup$ denotes the weak convergence (in an appropriate space).

\begin{theorem}\label{th1}
Let $q<\infty$. Let $\left\{f\e\right\}_\eps$ be a family of functions from $ L_2(\Omega\e)$ satisfying
\begin{gather}\label{f}\|f\e\|_{L_2(\Omega\e)}\leq C,\quad 
J_1\e f\e\rightharpoonup f_1\text{ in }L_2(\mathbb{R}),\quad J_2\e f\e\rightharpoonup f_2\text{ in }L_2(\mathbb{R})\text{ as }\eps\to 0,
\end{gather}
where  $f_1,f_2\in L_2(\mathbb{R})$. 
We set $u\e=(\mathcal{H}\e+\mu I)^{-1} f\e$, $\mu>0$.

Then
$$J_1\e u\e\rightharpoonup u_1\text{ in }H^1(\mathbb{R})\text{ as }\eps\to 0,$$
where the function $u_1$ belongs to $H^2(\mathbb{R})$ and is a solution of the problem
\begin{gather}\label{res_probl}
-u_1''+\mu\left(1+{q |B| \over q+\mu }\right)u_1=
f_1+{q|B|^{1/2} \over q+\mu  }f_2.
\end{gather}
Moreover 
$$J_2\e u\e\rightharpoonup u_2={q |B|^{1/2}  \over q +\mu  }u_1+ 
{1  \over q +\mu  }f_2\ \text{ in }L_2(\mathbb{R}).$$
\medskip
\end{theorem}

\begin{remark}\label{rem_f}
The typical example of a family $\{f\e\}_\eps$ satisfying 
\eqref{f} is as follows. Let $F=(f_1,f_2)\in [L_2(\mathbb{R})]^2$ be an arbitrary function. We introduce the function $f\e$ by
\begin{gather}\label{J*}
f\e(\x)=
\begin{cases}
\ds{1\over \sqrt{\eps}}f_1(x_1),& \x\in\Pi\e,\\
0,& \x\in T_j\e,\\
\ds{1\over \sqrt{\eps|B_j\e|}}\intl_{-\eps/2+x_j\e}^{\eps/2+x_j\e} f_2(x)\d x,&\x\in B_j\e.
\end{cases}
\end{gather}
It is easy to show that $f\e $ satisfies \eqref{f}.
\end{remark}

Theorem \ref{th1} can be rewritten by assigning to the problem \eqref{res_probl} some self-adjoint and positive operator. Namely, we introduce the operator $\mathcal{H}$ acting in $[L_2(\mathbb{R})]^2$ by    
\begin{gather}\label{H}
\mathcal{H}U= 
 \begin{pmatrix}-\ds{\d ^2/ \d x^2}+q|B| & -q|B|^{1/2} \\[1mm] - q|B|^{1/2} & q   \end{pmatrix} 
 \begin{pmatrix}u_1\\[1mm]u_2\end{pmatrix} ,\ U=(u_1,u_2),\quad \mathrm{dom}(\mathcal{H})=H^2(\mathbb{R})\times L_2(\mathbb{R}).
\end{gather} 
It is straightforward to show that $u_1$ solves \eqref{res_probl} and 
$u_2={q |B|^{1/2}  \over q +\mu  }u_1+ 
{1  \over q +\mu  }f_2$
iff
$$\mathcal{H}U+\mu U  = F,\text{ where }U=(u_1,  u_2),\ F=(f_1,  f_2).$$
Thus Theorem \ref{th1} claims  
$$J\e(\mathcal{H}\e+\mu I )^{-1} f\e \rightharpoonup (\mathcal{H}+\mu I )^{-1}F\text{ as }\eps\to 0\text{\quad provided }J\e f\e\rightharpoonup F\text{ in }[L_2(\mathbb{R})]^2,$$
where $J\e=(J_1\e, J_2\e):H^1(\Omega\e)\times L_2(\Omega\e)\to H^1(\mathbb{R})\times L_2(\mathbb{R})$. \bigskip

In the case $q=\infty$ one has the following result.

\begin{theorem}\label{th1+}
Let $q=\infty$. Let $\left\{f\e\right\}_\eps$ be a family of functions from $\in L_2(\Omega\e)$ satisfying \eqref{f}. We set $u\e=(\mathcal{H}\e+\mu I)^{-1} f\e$, $\mu>0$.
Then
$$
J_1\e u\e\rightharpoonup u_1\text{ in }H^1(\mathbb{R}),\quad
J_2\e u\e\rightharpoonup u_2\text{ in }L_2(\mathbb{R})
\text{\quad as }\eps\to 0,$$
where $u_1\in H^2(\mathbb{R})$, $u_2={|B|^{1/2}}u_1$ and 
\begin{gather}\label{res_probl+}
-u_1''+\mu\left(1+{|B|}\right)u_1=f_1+{|B|^{1/2}}f_2.
\end{gather}
\end{theorem}
\bigskip

In what follows we consider the case $q>0$ only. Our next goal is be to describe the behaviour of $\sigma(\mathcal{H}\e)$ as $\eps\to 0$. 

\begin{theorem}\label{th2} 
Let $L>0$ be an arbitrary number.
Then 
\begin{gather}\label{Hausdorff}
\mathrm{dist}_H\left(\sigma(\mathcal{H}\e)\cap [0,L],\, \sigma(\mathcal{H})\cap [0,L]\right)\to 0\text{ as }\eps\to 0,
\end{gather}
where $\mathrm{dist}_H(\cdot,\cdot)$ stays for the Hausdorff distance between two sets. \footnote{For two compact sets $X,Y\subset\mathbb{R}$ one has:
$\mathrm{dist}_H(X,Y)=\max\limits\left\{\sup\limits_{
x\in X }\inf\limits_{y\in Y} |x-y|;
\sup\limits_{
y\in Y }\inf\limits_{x\in X} |y-x|\right\}$.}
\end{theorem}

\begin{remark}\label{rem1}
The claim of Theorem \ref{th2} is equivalent to the fulfilment of the following conditions:
\begin{itemize}
 \item[(i)] Let the family $\left\{\lambda^{\eps}\in\sigma(\mathcal{H}\e)\right\}_{\eps}$ 
 have a convergent subsequence, i.e. $\lambda^{\eps}\to\lambda$ as $\eps=\eps_k\to 0$. Then $\lambda\in \sigma(\mathcal{H})$.

 \item[(ii)] Let $\lambda\in \sigma(\mathcal{H})$. Then there exists a family $\left\{\lambda^{\eps}\in\sigma(\mathcal{H}\e)\right\}_{\eps}$ such that   $\liml_{\eps\to 0}\lambda\e=\lambda$.

\end{itemize}

\end{remark}

It is easy to see that the spectrum of $\mathcal{H}$ has the following form:
\begin{gather}\label{sigmaH}
\sigma(\mathcal{H})=[0,\infty)\setminus \left(q ,q+q|B| \right).
\end{gather}
Indeed, for $\lambda\not= q $ the resolvent equation $\mathcal{H}U-\lambda U=F$ is
equivalent to
$$-u_1''-\rho(\lambda) u_1=f_1+{q|B|^{1/2}\over q-\lambda  }f_2,\quad u_2={q|B|^{1/2} \over q-\lambda  }u_1+{1\over q-\lambda}f_1,\quad\text{where }\rho(\lambda)=\lambda\left(1+{q|B|\over q-\lambda }\right),$$
whence, evidently, 
$\lambda\in\sigma(\mathcal{H})\setminus\{q \}$ iff $\rho(\lambda)\in\sigma(-{\d^2\over \d x^2}|_{L_2(\mathbb{R})})=[0,\infty)$. But  $\rho(\lambda)\in [0,\infty)$ iff $\lambda\in [0,\infty)\setminus [q    ,q+q|B| )$. Finally, $q \in \sigma(\mathcal{H})$ since $\sigma(\mathcal{H})$ is a closed set.

\begin{remark}\label{rm-bounded}

Theorems \ref{th1}-\ref{th2} (after a natural reformulation) remain valid if $\Omega$ is a \textit{bounded} strip: $\Omega=(0,l)\times(-\eps,0)$, $l>0$. In this case $\sigma(\mathcal{H}\e)$ is purely discrete. The limit problems \eqref{res_probl} and \eqref{res_probl+} are now considered on $(0,l)$ with Neumann conditions at the endpoints. If $q<\infty$  the spectrum of the corresponding limit operator $\mathcal{H}$ has the form
$$\sigma(\mathcal{H})=\left\{\lambda_k^-,\ k=1,2,3\dots\right\}\cup \left\{\lambda_k^+,\ k=1,2,3\dots\right\}\cup\{q\},$$
where the point $\lambda_k^{\pm}$ belong to the discrete spectrum, $q$ is the only point of the essential spectrum and 
$$\lambda_1^-=0,\quad\lambda_1^+=q+q|B|,\quad\lambda_k^-\nearrow q,\quad \lambda_k^+\nearrow\infty\text{\quad as }k\to\infty.$$ 
The proofs repeat  word-by-word the proofs for  the unbounded case.
\end{remark}\medskip

From Theorem \ref{th2} and \eqref{sigmaH} we conclude that $\sigma(\mathcal{H}\e)$ has at least one gap provided $\eps$ is small enough. Moreover, there is a gap converging to the interval $(q ,q+q|B| )$ as $\eps\to 0$. 
Unfortunately, Hausdorff convergence provides no information on the upper bound for the number of gaps, even within finite intervals: for example, the set 
$\sigma\e:=[0,L]\cap\bigg(\cupl_{k\in
\mathbb{N}}\left[\eps k,\eps(k+{1\over 2})\right]\bigg)$
converges to $[0,L]$ in the Hausdorff sense, but
the number of gaps in $\sigma\e$ tends to infinity as
$\eps\to 0$. Nevertheless, for our problem one can say more, namely, the following lemma take place.

\begin{lemma}\label{lm1}
Within an arbitrary compact interval $[0,L]$ the spectrum of $\mathcal{H}\e$ has at most one gap provided $\eps$ is small enough.
\end{lemma}
\medskip

Combining Theorem \ref{th2} and Lemma \ref{lm1}, we arrive at the main result of this work.
\begin{theorem}\label{th3}
Let $L>0$ be an arbitrary number. Then the spectrum of the operator $\mathcal{H}\e$ in $[0,L]$ has the following structure for $\eps$ small enough:
\begin{gather*}
\sigma(\mathcal{H}\e)\cap [0,L]=[0,L]\setminus
(\a\e,\b\e),
\end{gather*}
where the endpoints of the interval $(a\e,  \beta\e)$ satisfy 
\begin{gather}\label{main2}
\liml_{\eps\to 0}\a\e=q ,\quad \liml_{\eps\to 0}\b\e=q+q|B| .
\end{gather}
\end{theorem}

Theorems \ref{th1}-\ref{th2} as well as Lemma \ref{lm1} will be proven in the next sections.\bigskip

At the end of this section we introduce several notations, which further will be frequently used:
\begin{itemize}
\item $Y_j\e=\left\{\x\in\Pi\e:\ |x_1-x_j\e|<{\eps\over 2}\right\}$,

\item $S_j\e=\left\{\x\in\partial T_j\e:\ x_2=0\right\},$ 

\item $C_j\e=\left\{\x\in\partial T_j\e:\ x_2=h\e\right\}.$

\end{itemize}

The notation $\<u\>_{D}$ stays for the mean value of the function $u(\x)$ in the domain $D$, i.e.
$$\<u\>_{D}={1\over |D|}\intl_{D}u(\x)\d \x.$$
Also we keep the same notation if $D$ is a segment (e.g., $S_j\e$). In this case we integrate with respect to the natural coordinate on this segment, $|D|$ is its length.

\section{Proof of Theorem \ref{th1} \label{sec3}}

Let $\left\{f\e\right\}_\eps$ be a family of functions from $L_2(\Omega\e)$ satisfying \eqref{f},
$u\e=(\mathcal{H}\e+\mu I)^{-1} f\e$, $\mu>0$. One has the following standard estimates:
$$\|u\e\|_{L_2(\Omega\e)}\leq {1\over \mu}\|f\e\|_{L_2(\Omega\e)},\quad 
\|\nabla u\e\|^2_{L_2(\Omega\e)}\leq \|f\e\|_{L_2(\Omega\e)} \|u\e\|_{L_2(\Omega\e)},$$
whence, taking into account that $\|f\e\|_{L_2(\Omega\e)}\leq C$,
we obtain 
\begin{gather}\label{H1_est}
\|u\e\|^2_{H^1(\Omega\e)}\leq  C_1.
\end{gather}

Recall that the operators $J_j\e$, $j=1,2$ satisfy \eqref{iso0},
moreover, changing the order of integration with respect to $x_2$ and differentiation with respect to $x_1$, one can easily prove that for  $u\in H^1(\Omega\e)$   
\begin{gather}\label{J_H1}
\|(J_1\e u\e)'\|^2_{L_2(\mathbb{R})}\leq \|\partial_{x_1}u\e\|^2_{L_2(\Pi\e)}\leq  \|\nabla u\e\|^2_{L_2(\Omega\e)}.
\end{gather}
Then it follows from \eqref{iso0} (applied for $u\e$), \eqref{J_H1} that the families
$\left\{J_1\e u\e\right\}_\eps$ and  $\left\{ J_2\e u\e\right\}_\eps$ are uniformly bounded in  $H^1(\mathbb{R})$ and $L_2(\mathbb{R})$, respectively, and therefore there are a subsequence (for convenience, still indexed by $\eps$) and $u_1\in H^1(\mathbb{R})$, $u_2\in L_2(\mathbb{R})$ such that
\begin{gather}\label{conv12}
J_1\e u\e \rightharpoonup u_1\text{ in }H^1(\mathbb{R}),\quad
J_2\e u\e \rightharpoonup u_2\text{ in }L_2(\mathbb{R})\text{\quad as }\eps\to 0.
\end{gather}

Now, let us write the variational formulation of the resolvent equation \eqref{resolvent-eq}:
\begin{gather}\label{int_eq}
\intl_{\Omega\e}\bigg(\nabla u\e\cdot\nabla w +\mu u\e w \bigg) \d \x = 
 \intl_{\Omega\e} f\e w  \d \x,\ \forall w \in H^1(\Omega\e).
\end{gather}
Our strategy will be to plug into \eqref{int_eq} a specially chosen test-function $w=w\e$ and then pass to the limit as $\eps\to 0$ hoping to arrive at the equality $\mathcal{H}U+\mu U =F$  written in a weak form, where $\mathcal{H}$ is defined by \eqref{H}, $U=(u_1,u_2)$ and $F=(f_1,f_2)$. 

We choose this test-function as follows (below, as usual, $\x=(x_1,x_2)$):
\begin{gather}\label{we}
w=w\e(\x)=
\begin{cases}
\ds{1\over \sqrt{\eps}}\left(w_1(x_1)+\ds\suml_{j\in \mathbb{Z}}(w_1(x_j\e)-w_1(x_1))\varphi_j\e(\x)\right),& \x\in \Pi\e,\\
\ds {1\over h\e\sqrt{\eps}}\left({w_2(x_j\e)} -w_1(x_j\e) \right)x_2+{w_1(x_j\e)\over\sqrt{\eps}},& \x\in T_j\e,\\
\ds{1 \over \sqrt{\eps}}w_2(x_j\e),&\x\in B_j\e.
\end{cases}
\end{gather}
Here $w_1,w_2\in C^1_0(\mathbb{R})$ are arbitrary functions, 
$\varphi_j\e(\x)=\varphi\left({|x_1 -x_j\e|\over d\e}\right)$, where
$\varphi:\mathbb{R}\to\mathbb{R}$ is a smooth functions satisfying $\varphi(t)=1$ as $t\leq 1$ and $\varphi(t)=0$ as $t\geq 2$. Obviously $w\e\in H^1(\Omega\e)$ provided $d\e\leq {\eps\over 4}$ (this holds for $\eps$ small enough, see \eqref{ass_conc}).

We denote:
$$\I =\left\{j\in\mathbb{Z}:\ x_j\e\in\supp (w_1)\cup\supp(w_2)\right\}.$$
It is clear that 
\begin{gather}\label{number}
\suml_{j\in\I}1\leq C\eps^{-1}.
\end{gather}

Let us plug $w=w\e(\x)$  into \eqref{int_eq}. 
Since
$\mathrm{supp}(\varphi_j\e)\subset \overline{Y_j\e}$ and $w\e=\mathrm{const}.$ in $B_j\e$
we obtain from \eqref{int_eq}:
\begin{multline}\label{Iomega}
\underset{I_1}{\underbrace{\eps^{-{1\over 2}}\intl_{\Pi\e}\bigg(\nabla u\e(\x)\cdot \nabla w_1(x_1) +\mu u\e(\x) w_1(x_1)\bigg) \d \x}}
\\\displaybreak[3]
+\underset{I_2}{\underbrace{\eps^{-{1\over 2}}\suml_{j\in\mathbb{Z}}\intl_{Y_j\e}
\left(
\nabla u\e(\x)\cdot
\nabla \left(  (w_1(x_j\e)-w_1(x_1) )\phi_j\e(\x)\right)
\right)+\mu
u\e \big(w_1(x_j\e)-w_1(x_1)\big)\phi_j\e(\x)\bigg)\d \x}}+
\\
+\underset{I_3}{\underbrace{\suml_{j\in \mathbb{Z}}\intl_{T_j\e}\nabla u\e(\x)\cdot\nabla w\e(\x) \d\x}}+\underset{I_4}{\underbrace{\mu\suml_{j\in \mathbb{Z}}\intl_{T_j\e} u\e(\x) w\e(\x) \d \x }}
+\underset{I_5}{\underbrace{\eps^{-{1\over 2}}\mu\suml_{j\in \mathbb{Z}}w_2(x_j\e) \intl_{B_j\e} u\e(\x) \d \x}}=
\\
=
\underset{I_6}{\underbrace{\eps^{-{1\over 2}}\intl_{\Pi\e}f\e(\x)  w_1(x_1)  \d \x}}+\underset{I_7}{\underbrace{\eps^{-{1\over 2}}\suml_{j\in\mathbb{Z}}\intl_{Y_j\e} 
f\e(\x) \big(w_1(x_j\e)-w_1(x_1)\big)\phi_j\e(\x) \d \x}}
\\+\underset{I_8}{\underbrace{\suml_{j\in \mathbb{Z}}\intl_{T_j\e} f\e(\x) w\e(\x) \d \x }}
+\underset{I_9}{\underbrace{\eps^{-{1\over 2}}\suml_{j\in \mathbb{Z}}w_2(x_j\e) \intl_{B_j\e} f\e(\x) \d \x}}.
\end{multline}

Let us analyse step-by-step the terms $I_j$, $j=1,\dots,9$.\medskip

\noindent (${I_1}$) Using \eqref{conv12} and the definition of the operator $J_1\e$ we obtain:
\begin{multline}\label{Iomega1}
I_1= \eps^{-{1\over 2}}
\intl_{-\infty}^\infty \intl_{-\eps}^0 \left({\partial u\e\over\partial x_1}(x_1,x_2) {\d w_1\over \d x_1}(x_1)+\mu u\e w_1\right)  \d x_2 \d x_1\\ =
\intl_{\mathbb{R}}  \big( (J\e_1 u\e)'w_1'+\mu (J_1\e u\e) w_1\big)  \d x\underset{\eps\to 0}\to 
\intl_{\mathbb{R}}  \big( u_1'w_1'+\mu u_1 w_1\big)  \d x.
\end{multline}\medskip

\noindent (${I_2}$) 
One has the following properties (below $\varphi_j\e$ is regarded as a function of $x\in\mathbb{R}$):
\begin{gather*}
\supp\left( \left(w_1(x_j\e)-w_1\right)\phi_j\e \right)\subset
 \left\{x\in\mathbb{R}:\ |x-x_j\e|<2d\e\right\},\quad 
\left|\left(\left(w_1(x_j\e)-w_1\right)\phi_j\e\right)'\right|+\left|\left(w_1(x_j\e)-w_1\right)\phi_j\e\right|\leq C.
\end{gather*}
Using them, \eqref{ass_conc}, \eqref{H1_est} and \eqref{number} we obtain easily:
\begin{gather}\label{Iomega2}
|I_2|\leq C\eps^{-{1\over 2}}\|u\e\|_{H^1(\Pi\e)}\sqrt{\suml_{j\in \I}\eps d\e}\leq C_1 \sqrt{d\e\over\eps}\underset{\eps\to 0} \to 0.
\end{gather}\medskip

\noindent (${I_3}$) 
Integrating by parts and taking into account that $\Delta w\e=0$ in $T_j\e$, we obtain:
\begin{multline}\label{Iomega3add1}
I_3=
\suml_{j\in\mathbb{Z}}\intl_{x_j\e-{d\e\over 2}}^{x_j\e+{d\e\over 2}}\left(u\e(x_1,h\e) {\partial w\e \over\partial x_2}(x_1,h\e) -u\e(x_1,0) {\partial w\e \over\partial x_2}(x_1,0)\right) \d x_1\\=
{d\e\over h\e\sqrt{\eps}}\suml_{j\in\mathbb{Z}}\left(\langle u\e\rangle_{S_j\e} -  {\langle u\e\rangle_{C_j\e}}\right)\big(w_1(x_j\e)-w_2(x_j\e)\big).
\end{multline}

Let us introduce the operator $Q\e: C_0^1(\mathbb{R})\to L_2(\mathbb{R})$ by 
\begin{gather*}
(Q\e w)(x)=\suml_{j\in \mathbb{Z}}w(x_j\e)\chi_j\e(x)
\end{gather*}
(recall that $\chi_j(x)$ is the indicator function of the interval $\left(x_j\e-{\eps\over 2},x_j\e+{\eps\over 2}\right)$). 
It is easy to show that
\begin{gather}\label{Q}
\forall w\in C^1(\mathbb{R}):\quad Q\e w\underset{\eps\to 0}\to w\text{ in }L_2(\mathbb{R}).
\end{gather}

With this operator one can rewrite \eqref{Iomega3add1} as
\begin{multline}\label{Iomega3add2}
I_3=
{d\e\over  h\e\sqrt{\eps}}\suml_{j\in\mathbb{Z}}\left(\langle u\e\rangle_{Y_j\e} -{ \langle u\e\rangle_{B_j\e} }\right)\big(w_1(x_j\e)-w_2(x_j\e)\big)+\delta(\eps)\\=
{d\e\over h\e\eps^2}\intl_{\mathbb{R}}\left(J_1\e u\e - |B|^{-1/2}J_2 u\e\right)\left(Q\e w_1- Q\e w_2\right)\d x+\delta(\eps),
\end{multline}
where $\delta(\eps)=\ds\suml_{j\in\I}{d\e\over h\e\sqrt{\eps}}\left( \langle u\e\rangle_{S_j\e} - \langle u\e\rangle_{Y_j\e}-\langle u\e\rangle_{C_j\e} +\langle u\e\rangle_{B_j^{\eps}}\right)\left(w_1(x_j\e)-w_2(x_j\e)\right).$ The last equality in \eqref{Iomega3add2} follows simply from the definitions of the operators $J_1\e$, $J_2\e$, $Q\e$.

To estimate the reminder $\delta(\eps)$ we need an additional lemma.
\begin{lemma}\label{lm2}
One has:
\begin{gather}\label{lm2est1}
\forall u\in H^1(Y_j\e):\quad \left|\<u\>_{S_j\e}-\<u\>_{Y_j\e}\right|\leq C \sqrt{|\ln d\e|}\|\nabla u\|_{L_2(Y_j\e)},\\
\label{lm2est2}\forall u\in H^1(B_j\e):\quad 
\left|\<u\>_{C_j\e}-\<u\>_{B_j^{\eps}}\right|\leq C\sqrt{|\ln d\e|}\|\nabla u\|_{L_2(B_j\e)}.
\end{gather}
\end{lemma}

\begin{proof}
For an arbitrary $u\in H^1(Y_j\e)$ one has the  following estimate (see \cite[Lemma~3.1]{CK15}):
\begin{gather}\label{est1+}
\left|\<u\>_{S_j\e}-\<u\>_{\Gamma_j\e}\right|\leq C_1 \sqrt{|\ln d\e|}\|\nabla u\|_{L_2(Y_j\e)},
\end{gather}
where $\Gamma_j\e=\left\{\x\in\partial Y_j\e:\ x_2=0\right\}$.
Moreover, using the trace  and   Poincar\'{e} inequalities we obtain
\begin{multline}\label{est1++}
\left|\<u\>_{\Gamma_j\e}-\<u\>_{Y_j\e}\right| = 
\left| \eps^{-1}\intl_{x_j\e-{\eps\over 2}}^{x_j\e+{\eps\over 2}}\left(u(x_1,0)-\<u\>_{Y_j\e}\right) \d x_1\right|
\leq \eps^{-1/2}\sqrt{\intl_{x_j\e-{\eps\over 2}}^{x_j\e+{\eps\over 2}} \left|u(x_1,0)-\<u\>_{Y_j\e}\right|^2\d x_1}
\\ \leq C\eps^{-1/2}\sqrt{\eps\|\nabla u\|^2_{L_2(Y_j\e)}+\eps^{-1}\left\|u-\<u\>_{Y_j\e}\right\|^2_{L_2(Y_j\e)}}\leq 
C_1\|\nabla u\|_{L_2(Y_j\e)}.
\end{multline}
Then \eqref{lm2est1} follow from \eqref{est1+}-\eqref{est1++}. 
The proof of estimate \eqref{lm2est2} is similar (instead of $\Gamma_j\e$ one should use the set $\left\{\x\in \mathbb{R}^2:\ x_2=h\e,\ |x_1-x_j\e|<R\eps\right\}$, $R$ is defined in \eqref{assB2}).
\end{proof}

\begin{remark}
Using similar arguments (cf. \cite[Lemma~3.1]{CK15}) one can also prove the estimate
\begin{gather}\label{strong} 
\forall u\in H^1(Y_j\e):\quad 
\|u\|^2_{L_2(S_j\e)} \leq 
d\e{\eps^{-2}}\|u\|^2_{L_2(Y_j\e)}+  C{d\e|\ln d\e|}\|\nabla u\|^2_{L_2(Y_j\e)}.
\end{gather}
We will use it later in the proof of Theorem \ref{th2}.
\end{remark}\medskip

Using Lemma \ref{lm2} and taking into account \eqref{ass1}, \eqref{ass3}, \eqref{H1_est} and \eqref{number} we get:
\begin{multline}\label{delta_est}
|\delta(\eps)|\leq  {d\e\over h\e\sqrt{\eps}}
\sqrt{\suml_{j\in\I}\left(\left| \langle u\e\rangle_{S_j\e} - \langle u\e\rangle_{Y_j\e}\right|^2+\left|\langle u\e\rangle_{B_j^{\eps}}-\langle u\e\rangle_{C_j\e}\right|^2\right)}\\\times\sqrt{\suml_{j\in\I}\left(\maxl_{x\in \mathbb{R}}|w_1(x)|^2+\maxl_{x\in \mathbb{R}}|w_2(x)|^2\right)}\leq C\sqrt{\eps^2 |\ln d\e|} {d\e\over h\e \eps^2}\|\nabla u\e\|_{L_2(\Omega\e)}\underset{\eps\to 0}\to 0.
\end{multline}

Combining \eqref{Iomega3add2} and \eqref{delta_est} and taking into account
 \eqref{ass3}, \eqref{conv12},  \eqref{Q}, we obtain: 
\begin{gather}\label{Iomega3}
I_3\underset{\eps\to 0}\to q|B|\intl_{\mathbb{R}}\left(u_1-|B|^{-1/2}u_2\right)\left(w_1-w_2\right)\d x.
\end{gather}\medskip

\noindent (${I_4}$) Taking into account that $\maxl_{\x\in T_j\e}|w\e(\x)|\leq C\eps^{-1/2}$ and using \eqref{ass2}, \eqref{ass3} and \eqref{number} we estimate:
\begin{gather}\label{Iomega4}
|I_4|\leq C\|u\|_{L_2(\cup_j T_j\e)}\sqrt{\suml_{j\in\I} |T_j\e|\eps^{-1}}\leq C_1 h\e \sqrt{d\e\over h\e\eps^2} \underset{\eps\to 0}\to 0.
\end{gather}\medskip

\noindent (${I_5}$) Using \eqref{conv12}, \eqref{Q} we arrive at
\begin{gather}\label{Iomega5}
I_5= 
\mu|B|^{1/2}\intl_{\mathbb{R}} (J_2\e u\e)(Q\e w_2)\d x\underset{\eps\to 0}\to \mu|B|^{1/2}\intl_{\mathbb{R}} u_2 w_2\d x.
\end{gather}
Here the first equality follows simply from the definitions for the operators $J_2\e$ and $Q\e$.\medskip

\noindent (${I_6}$)-(${I_9}$) By virtue of arguments similar to those ones in (${I_1}$), (${I_2}$), (${I_4}$), (${I_5}$) and taking into account \eqref{f} we obtain the following asymptotic behavior for the terms in the right-hand-side of \eqref{Iomega}:
\begin{gather}\label{Iomega6-9}
I_6\e\underset{\eps\to 0}\to \intl_{\mathbb{R}} f_1 w_1\d x,\quad
I_7\e\underset{\eps\to 0}\to 0,\quad
I_8\e\underset{\eps\to 0}\to 0,\quad
I_9\underset{\eps\to 0}\to \mu|B|^{1/2}\intl_{\mathbb{R}} f_2 w_2\d x.
\end{gather}

Finally, combining \eqref{Iomega}-\eqref{Iomega2}, \eqref{Iomega3}-\eqref{Iomega6-9}, we arrive at the equality
\begin{multline}\label{int_eq_final1}
\intl_{\mathbb{R}}u'_1 w'_1 \d x+q|B|\intl_{\mathbb{R}}\left(u_1-|B|^{-1/2} u_2\right)
\left(w_1-w_2\right)\d x+\mu \intl_{\mathbb{R}}\left(u_1 w_1+|B|^{1/2} u_2 w_2 \right)\d x\\=
\intl_{\mathbb{R}}\left(f_1 w_1+|B|^{1/2} f_2 w_2 \right)\d x,
\end{multline}
which is valid for an arbitrary $w_1,w_2\in C^1_0(\mathbb{R})$ (and therefore, 
by the density arguments, for an arbitrary $w_1\in H^1(\mathbb{R})$ and $w_2\in L_2(\mathbb{R})$).

Taking $w_1\equiv 0$ in \eqref{int_eq_final1} we get
$
\intl_{\mathbb{R}}\left(-q|B|u_1+ q|B|^{1/2} u_2 +\mu   |B|^{1/2} u_2 - |B|^{1/2} f_2 \right)w_2 \d x$, $\forall w_2\in L_2(\mathbb{R})$,
whence
\begin{gather}\label{u2}
u_2={q |B|^{1/2}  \over q +\mu   }u_1+ {1   \over q +\mu   }f_2.
\end{gather}
Then, taking $w_2\equiv 0$ in \eqref{int_eq_final1} and using \eqref{u2}, we arrive at
\begin{gather*}
\intl_{\mathbb{R}}u'_1 w'_1 \d x+\mu\intl_{\mathbb{R}}\left(1+{q|B|\over q+\mu }\right) u_1 w_1   \d x=
\intl_{\mathbb{R}}\left(f_1  + { q|B|^{1/2}\over q+\mu } f_2\right)w_1 \d x,\ \forall w_1\in H^1(\mathbb{R}),
\end{gather*}
whence, $u_1$ belongs to $H^2(\mathbb{R})$ and is a solution to the problem \eqref{res_probl}.

Finally, since the problem \eqref{res_probl} has the unique solution and $u_2$ is uniquely determined  by $u_1$ via \eqref{u2}, then \eqref{conv12}  hold for the whole sequence $u\e$. Theorem \ref{th1} is proved.

\section{Proof of Theorem \ref{th1+} \label{sec4}}

Via the same arguments as in the proof of Theorem \ref{th1} we conclude that there is a subsequence (for convenience, still indexed by  $\eps$) and  $u_1\in H^1(\mathbb{R})$, $u_2\in L_2(\mathbb{R})$ such that \eqref{conv12} holds.

For an arbitrary $w\in H^1(\Omega\e)$ one has the equality \eqref{int_eq}.
We plug into this equality the function $w=w\e(\x)$ defined by \eqref{we}, but with 
$w_1(x)=w_2(x)$. In this case the terms $I_3$ and $I_8$ (see \eqref{Iomega}) are equal to zero, while the behaviour of the rest terms  is independent of whether $q$ is finite or not.
Thus, passing to the limit in \eqref{int_eq} we arrive at the equality
\begin{gather}\label{int_eq_final1+}
\intl_{\mathbb{R}}u'_1 w'_1 \d x+\mu \intl_{\mathbb{R}}\left(u_1 +|B|^{1/2} u_2 \right)w_1 \d x=
\intl_{\mathbb{R}}\left(f_1 +|B|^{1/2} f_2 \right)w_1\d x,
\end{gather}
which is valid for an arbitrary $w_1 \in H^1(\mathbb{R})$.

It remains to show that $u_2=|B|^{1/2}u_1$ (then, evidently, \eqref{int_eq_final1+} will imply $u_1\in H^2(\mathbb{R})$ and \eqref{res_probl+}). One has, using the definitions of the operators $J_1\e$ and $J_2\e$:
\begin{multline}\label{j2-j1}
\left\| J_2\e u\e - |B|^{1/2}J_1\e u\e\right\|_{L_2(\mathbb{R})}^2=
\suml_{j\in\mathbb{Z}}\intl_{x_j\e-\eps/2}^{x_j\e+\eps/2}\left|{|B|^{1/2}  \eps^{1/2}}\<u\e\>_{B_j\e} - {|B|^{1/2}  \eps^{-1/2}}\intl_{-\eps}^0 u\e(x_1,x_2)\d x_2 \right|^2 \d x_1\\=\eps^{-1}|B|\suml_{j\in\mathbb{Z}}\intl_{x_j\e-\eps/2}^{x_j\e+\eps/2}\left| \intl_{-\eps}^0\left(\<u\e\>_{B_j\e}- u\e(x_1,x_2)\right)\d x_2 \right|^2 \d x_1\leq
|B|\suml_{j\in\mathbb{Z}}\left\|\<u\e\>_{B_j\e}- u\e\right\|_{L_2(Y_j\e)}^2\\ 
\leq 4|B|\suml_{j\in\mathbb{Z}}\left(\left\|\<u\e\>_{Y_j\e}- u\e\right\|_{L_2(Y_j\e)}^2+
\eps^2\left(\left|\<u\e\>_{S_j\e}- \<u\e\>_{Y_j\e}\right|^2+ \left|\<u\e\>_{C_j\e}- \<u\e\>_{S_j\e}\right|^2+ \left|\<u\e\>_{B_j\e}- \<u\e\>_{C_j\e}\right|^2\right)\right).
\end{multline}
The following simple estimate holds (cf. \citep[Lemma 3.2]{CK15}):
\begin{gather}\label{s-c}
\forall u\in H^1(T_j\e):\ \left|\<u \>_{C_j\e}- \<u \>_{S_j\e}\right|^2\leq C{h\e  (d\e)^{-1}}\|\nabla u \|^2_{L_2(T_j\e)}.
\end{gather}
Then, using \eqref{lm2est1}, \eqref{lm2est2}, \eqref{s-c} and the Poincar\'{e} inequality, we obtain from \eqref{j2-j1}:
\begin{multline}\label{j2-j1+}
\left\| J_2\e u\e - |B|^{1/2}J_1\e u\e\right\|_{L_2(\mathbb{R})}^2\leq 
C_1\eps^2 \|\nabla u\e\|^2_{L_2(\cupl_{j\in\mathbb{Z}}Y_j\e)}\\+
C_2\eps^2|\ln d\e|\left( \|\nabla u\e\|^2_{L_2(\cupl_{j\in\mathbb{Z}}(B_j\e\cup Y_j\e))}\right)+
C_3\eps^2 h\e (d\e)^{-1} \|\nabla u\e\|^2_{L_2(\cupl_{j\in\mathbb{Z}}T_j\e)}\to 0\text{ as }\eps\to 0
\end{multline}
(here the the right-hand-side tends to zero due to \eqref{ass1} and \eqref{ass3} (recall, that $q=\infty$)). 
Finally, in view the Rellich embedding theorem, the weak convergence of $J_1\e u\e$ to $u_1$ in $H^1(\mathbb{R})$ implies
\begin{gather}\label{strongLL}
\forall L>0:\quad \|J_1\e u\e-u_1\|_{L_2(-L,L)}\to 0\text{ as }\eps\to 0. 
\end{gather}
From \eqref{j2-j1} and \eqref{strongLL} we deduce $u_2=|B|^{1/2}u_1$.

Since the problem \eqref{res_probl+} has the unique  solution and $u_2$ is uniquely determined  by $u_1$, then \eqref{conv12}  hold for the whole sequence $u\e$. 
Theorem \ref{th1+} is proved.

\section{Proof of Theorem \ref{th2} \label{sec5}}

Recall, that we have to check the fulfilment of the properties (i)-(ii)
(see Remark \ref{rem1}).

\subsection{Proof of the property (i)}

Let $\lambda\e\in\sigma(\mathcal{H}\e)$ and $\lambda\e\to \lambda$ as $\eps=\eps_k\to 0 $. We have to show that $\lambda\in\sigma(\mathcal{H})$. 

In what follows we will use the notation $\eps$ taking in mind $\eps_k$.
To simplify the presentation we suppose that $\eps$ takes values in the discrete set $ \left\{\eps:\ \eps^{-1}\in\mathbb{N}\right\}$. 
The general case needs slight modifications.

We denote
\begin{itemize} 

\item $\mathcal{N}\e=\{1,2,\dots,\eps^{-1}\}$,

\item $\widetilde\Pi^{\eps}=\left\{\x \in \mathbb{R}^2:\ 0<x_1<1,\ -\eps< x_2< 0\right\}$,

\item $\widetilde \Omega\e=\widetilde \Pi\e\cup\left(\ds\cupl_{j\in\mathcal{N}\e}\left(T_j\e\cup B_j\e\right)\right)$,

\end{itemize}
It is clear that the set $\widetilde\Omega\e$ is a period cell for $\Omega\e$, namely 
$$\Omega\e=\cupl_{k\in\mathbb{Z}}\overline{(\widetilde\Omega\e + k)},\quad (\widetilde\Omega\e + k)\cap (\widetilde\Omega\e + l)=\varnothing\text{ for }k,l\in\mathbb{Z},\ k\not= l.$$

Let $\varphi\in \mathbb{R}\backslash(2\pi\mathbb{Z})$. In the space $L_2(\widetilde\Omega\e)$ we introduce the sesquilinear form $\mathfrak{h}^{\varphi,\eps}$ by
\begin{gather*}
\mathfrak{h}^{\varphi,\eps}[u,v]=\intl_{\widetilde\Omega\e}\nabla u\cdot\overline{\nabla v} \d \x,\quad
\mathrm{dom}(\mathfrak{h}^{\varphi,\eps})=\left\{u\in H^1(\widetilde\Omega\e):\ u(1,x_2)=\exp(i\varphi)u(0,x_2)\text{ for }x_2\in (-\eps,0)\right\}.
\end{gather*}
We denote by $\mathcal{H}^{\varphi,\eps}$ the operator  associated with this form. One has $\mathcal{H}^{\varphi,\eps}u=-\Delta u$ in the generalized sense;
the function $u\in \mathrm{dom}(\mathcal{H}^{\varphi,\eps})$ satisfies (in a sense of traces)
$$u(1,x_2)=\exp(i\varphi)u(0,x_2),\quad {\partial u\over\partial x_1}(1,x_2)=\exp(i\varphi){\partial u\over\partial x_1}(0,x_2)\quad\text{for }x_2\in (-\eps,0).$$

The spectrum of $\mathcal{H}^{\varphi,\eps}$ is purely discrete.
We denote by $\left\{\lambda_k^{\varphi,\eps}\right\}_{k\in\mathbb{N}}$ the sequence of eigenvalues of $\mathcal{H}^{\varphi,\eps}$ arranged in the ascending order and repeated according to 
their multiplicity. By $\left\{u^{\varphi,\eps}_k\right\}_{k\in\mathbb{N}}$ we denote the corresponding sequence of eigenfunctions  such that $(u^{\varphi,\eps}_k,u^{\varphi,\eps}_l)_{L_2(\widetilde{\Omega}\e)}=\delta_{kl}$.

Using Floquet-Bloch theory (see, e.g., \citep{DLPSW11,K93,RS72}) we deduce the following relationship between the spectra of $\mathcal{H}\e$ and $\mathcal{H}^{\varphi,\eps}$:
\begin{gather}\label{Floquet}
\sigma(\mathcal{H}\e)=\cupl_{k\in\mathbb{N}}
\left\{ \lambda_k^{\varphi,\eps}:\ \varphi\in \mathbb{R}\backslash(2\pi\mathbb{Z})\right\}.
\end{gather}
For fixed $k\in\mathbb{N}$ the set $\left\{ \lambda_k^{\varphi,\eps}:\ \varphi\in \mathbb{R}\backslash(2\pi\mathbb{Z})\right\}$ is a compact interval.\smallskip

Since $\lambda\e\in\sigma(\mathcal{H}\e)$ then in view of \eqref{Floquet} there is $\varphi\e\in  \mathbb{R}\backslash(2\pi\mathbb{Z})$, $k\e\in\mathbb{N}$ such that $\lambda\e=\lambda_{k\e}^{\varphi,\eps}$. By 
 $u\e=u_{k\e}^{\varphi\e,\eps}$ we denote the corresponding eigenfunction. One has:
\begin{gather}\label{normal}
 \|u\e\|_{L_2(\widetilde\Omega\e)}=1\quad\text{(and, consequently, $\|\nabla u\e\|^2_{L_2(\widetilde\Omega\e)}=\lambda\e$)}.
\end{gather}

One can extract a convergent subsequence (for convenience, still indexed by $\eps$)
\begin{gather}\label{phi}
\varphi\e\to\varphi\in  \mathbb{R}\backslash(2\pi\mathbb{Z}).
\end{gather}

We define the operators $\widetilde J_1\e: H^1(\widetilde\Pi\e)\to H^1(0,1)$ and
$\widetilde J_2\e : L_2(\cupl_{j\in\mathcal{N}\e} B_j\e)\to L_2(0,1)$ by  \eqref{j1}-\eqref{j2} with $\mathcal{N}\e$ instead of $\mathbb{Z}$.
Via the same arguments as in the proof of Theorem \ref{th1}, we conclude from \eqref{normal} that the families $\left\{J_1\e u\e\right\}_\eps$ and  $\left\{ J_2\e u\e\right\}_\eps$ are uniformly bounded in  $H^1(0,1)$ and $L_2(0,1)$, respectively, and therefore there exists a subsequence (again indexed by $\eps$) and $u_1\in H^1(0,1)$, $u_2\in L_2(0,1)$ such that 
\begin{gather}\label{conv1+}
\widetilde J_1\e u\e \rightharpoonup u_1\text{ in }H^1(0,1),\\\label{conv2+}
\widetilde J_2\e u\e \rightharpoonup u_2\text{ in }L_2(0,1).
\end{gather}
Moreover, using the trace theorem, we obtain 
\begin{gather}\label{conv3+}
(\widetilde J_1\e u\e)(0) \to u_1(0),\quad (\widetilde J_1\e u\e)(1) \to u_1(1).
\end{gather}
It is clear that $(\widetilde J_1\e u\e)(1) =\exp(i \varphi\e )  (\widetilde J_1\e u\e)(0) $, whence, in view of \eqref{phi}  and \eqref{conv3+},  
\begin{gather}\label{phi+}
u_1(1)=\exp(i \varphi)  u_1(0).
\end{gather}
\smallskip

We start from the case $$u_1\not = 0.$$
We need the following analogue of Theorem \ref{th1}.
\begin{lemma}\label{lm+}
Let the family $\left\{f\e\in L_2(\widetilde\Omega\e)\right\}_\eps$ satisfy
\begin{gather}\label{f+}\|f\e\|_{L_2(\widetilde\Omega\e)}\leq C,\quad 
\widetilde J_1\e f\e\rightharpoonup f_1\text{ in }L_2(0,1),\quad \widetilde J_2\e f\e\rightharpoonup f_2\text{ in }L_2(0,1)\text{ as }\eps\to 0.
\end{gather}
We set $v\e_{f\e}=(\mathcal{H}^{\varphi\e,\eps}+\mu)^{-1}f\e$, where $\mu>0$. 
Then
$$\widetilde J_1\e v\e_{f\e}\rightharpoonup v_1\text{ in }H^1(0,1)\text{ as }\eps\to 0,$$
where $v_1\in H^2(0,1)$ satisfies 
$v_1(1)=\exp({i \varphi})  v_1(0),\quad v_1'(1)=\exp({i \varphi})  v_1'(0)$
and solves the problem \eqref{res_probl} on the interval $(0,1)$. Moreover 
$$\widetilde J_2\e u\e\rightharpoonup v_2={q |B|^{1/2}  \over q +\mu   }v_1+ {1 \over q +\mu  }f_2\ \text{ in }L_2(0,1).$$
\end{lemma}

\begin{proof}
The proof is similar to the proof of Theorem \ref{th1}. 
The only essential difference is that the test-function $w=w\e(\x)$ defined by \eqref{we} have to be modified in order to meet $\varphi\e$-periodic boundary conditions. Namely, let $w_1\in C^\infty(0,1)$ satisfy $w_1(1)=\exp({i \varphi})w_1(0)$. We introduce   $w\e_1\in C^\infty(0,1)$ by  
\begin{gather}\label{bc}
w_1\e(x)=w_1(x)\left((\exp(i\varphi\e-i\varphi)-1)x+1\right).
\end{gather}
Clearly $w_1\e(x)$ satisfies $w_1\e(1)=\exp({i\varphi\e})w_1\e(0)$ and
\begin{gather}\label{appr1}
w_1\e\to w_1\text{ in }C^1(0,1)\text{ as }\eps\to 0.
\end{gather}
Finally, we define the function $w$ by formula \eqref{we} with $w\e_1(x)$ instead of $w_1(x)$. 
In view of \eqref{bc} $w\e\in\mathrm{dom}(\mathfrak{h}^{\varphi\e,\eps})$.
Then we plug the function $w\e$ into  the equality
\begin{gather*}
\intl_{\widetilde\Omega\e}\bigg(\nabla v\e_{f\e}\cdot\nabla w\e +\mu v\e_{f\e} w\e \bigg) \d \x = 
 \intl_{\widetilde\Omega\e} f\e w\e  \d \x
\end{gather*}
and pass to the limit as $\eps\to 0$. Using the same arguments as in the proof of Theorem \ref{th1} (with account of  \eqref{appr1}) we arrive at the statement of the lemma.
\end{proof}

We choose $f\e=(\lambda+\mu)u\e$. It is clear that in this case $v\e_{f\e}=u\e$. 
Due to \eqref{conv1+}-\eqref{conv2+} conditions \eqref{f+} hold true. Then by Lemma \ref{lm+} $u_1$ belongs to $H^2(0,1)$ and satisfies (additionally to \eqref{phi+})
\begin{gather}\label{u1cond}
u_1'(1)=\exp(i \varphi)  u'_1(0)
\\\label{u1cond+}
-u_1''+\mu\left(1+{q |B| \over q+\mu  }\right)u_1=(\lambda+\mu)u_1+{q|B|^{1/2} \over q+\mu  }(\lambda+\mu)u_2,\quad
u_2={q |B|^{1/2}  \over q +\mu   }u_1+ {1   \over q +\mu  }(\lambda+\mu)u_2,
\end{gather}
From \eqref{u1cond+}, via simple calculations, we obtain the following equation for $u_1$:
\begin{gather}\label{u1eq}
-u_1''=\rho(\lambda) u_1,\text{ where }\rho(\lambda)=\lambda\left(1+{q|B|\over q-\lambda }\right).
\end{gather}

Since $u_1\not=0$ \eqref{phi+}, \eqref{u1cond}, \eqref{u1eq} imply that
$\rho(\lambda)\in\sigma(-{\d^2\over \d x^2}|_{L_2(\mathbb{R})})=[0,\infty)$ or, equivalently, 
$\lambda\in[0,\infty)\setminus (q ,q+q|B| ).$ 
Then due to \eqref{sigmaH} $\lambda\in\sigma(\mathcal{H})$.\bigskip

Now, we inspect the case $$u_1=0.$$ We  show that in this case $\lambda=q $ and hence (see \eqref{sigmaH}) $\lambda\in\sigma(\mathcal{H})$.

Recall that $\lambda\e=\lambda^{\varphi\e,\eps}_{k\e}$, $u\e=u^{\varphi\e,\eps}_{k\e}$. 
We express the eigenfunction $u\e$ in the form
\begin{gather*}
u\e=v\e-w\e+\delta\e,
\end{gather*}
where 
\begin{gather*}
v\e(\x)=
\begin{cases}
0,&\x\in \widetilde\Pi\e,\\
{ \langle u\e\rangle_{B_j\e} }(h\e)^{-1} x_2,&\x\in T_j\e,\\
{ \langle u\e\rangle_{B_j\e} },&\x\in B_j\e,
\end{cases}\quad\qquad w\e(\x)=\suml_{k=1}^{k\e-1}(v\e,u^{\varphi\e,\eps}_k)_{L_2(\widetilde\Omega\e)} u^{\varphi\e,\eps}_k(\x)
\end{gather*}
and $\delta\e$ is a remainder term. 
It is clear that  
\begin{gather}\label{g_prop1}
v\e,w\e\in \mathrm{dom}(\mathfrak{h}^{\varphi\e,\eps})\text{\quad and\quad }
v\e-w\e\in \left(\mathrm{span}\left\{u_1^{\varphi\e,\eps},\dots,u_{k\e-1}^{\varphi\e,\eps}\right\}\right)^\perp. 
\end{gather}
 
\begin{lemma}\label{lm4}
One has for each $u\in H^1(T_j\e\cup Y_j\e)$:
\begin{gather}\label{est4}
\|u\|^2_{L_2( T_j\e)}\leq C  \left( (h\e)^2\|u\|^2_{L_2(Y_j\e)}+d\e|\ln d\e|h\e  \|\nabla u\|^2_{L_2(Y_j\e)}+  (h\e)^2\|\nabla u\|^2_{L_2(  T_j\e)}\right).
\end{gather}
\end{lemma}

\begin{proof}
By the density arguments it is enough to prove the lemma only for smooth functions. Let $u$ be an arbitrary function from $C^1(\overline{T_j\e\cup Y_j\e})$.
Let $\x=(x_1,x_2)\in T_j\e$, $\mathbf{y}=(x_1,0)\in S_j\e$.
One has
\begin{gather*}
u(\x)=u(\mathbf{y})+\intl_{0}^{x_2}{\partial
u(\xi(\tau))\over\partial\tau}\d\tau,\text{ where }\xi(\tau)=(x_1,\tau),
\end{gather*}
whence, using \eqref{strong}, we obtain: 
\begin{multline}\label{tr8}\hspace{-3mm}
\|u\e\|^2_{L_2(T_j\e)}=
\intl_{0}^{h\e} \intl_{x_j\e-d\e/2}^{x_j\e+d\e/2} |u(x_1,x_2)|^2  \d x_1 \d x_2 \leq 2h\e \intl_{x_j\e-d\e/2}^{x_j\e+d\e/2} |u(x_1,0)|^2\d x_1 + 2 ( h\e)^2\intl_0^{h\e}\intl_{x_j\e-d\e/2}^{x_j\e+d\e/2}| \partial_{x_2} u(x_1,x_2)|^2\d x_1\d x_2\\\leq C h\e \left(d\e\eps^{-2}\|u\|_{L_2(Y_j\e)}^2+d\e |\ln d\e| \|\nabla u \|^2_{L_2(Y_j\e)}  \right)  + 2(h\e)^2 \|\nabla u\e\|^2_{L_2(T_j\e)}.
\end{multline} 
From \eqref{tr8}, taking into account that $d\e\eps^{-2}=\mathcal{O}(h\e)$ (see \eqref{ass3} for $q<\infty$), we arrive at \eqref{est4}.
\end{proof}

Estimate \eqref{est4} yields
\begin{gather}\label{uT}
\|u\e\|_{L_2\left(\cup_{j\in \mathcal{N}\e} T_j\e\right)}\to 0\text{ as }\eps\to 0.
\end{gather}
Also, one has the following Poincar\'e inequality:
\begin{gather}\label{uB}
\suml_{j\in\mathcal{N}\e}\|u\e-{ \langle u\e\rangle_{B_j\e} }\|^2_{L_2(B_j\e)}\leq C\eps^2
\|\nabla u\e \|^2_{L_2(B_j\e)}\to 0\text{ as }\eps\to 0.
\end{gather}
Since $u_1=0$, then, using \eqref{iso1} (evidently, it  holds with $\widetilde{\Pi}\e$ and $(0,1)$ instead of $\Pi\e$ and $\mathbb{R}$), we get
\begin{gather}
\label{uPi}
\|u\e\|^2_{L_2(\widetilde\Pi\e)}\leq \|J_1\e u\e\|^2_{L_2(0,1)}+C\eps^2 \|\nabla u\e\|^2_{L_2(\widetilde\Pi\e)}\to 0\text{ as }\eps\to 0.
\end{gather}

Since $\|u\e\|_{L_2(\widetilde{\Omega}\e)}=1$ then
\begin{gather*}
1=
\|u\e\|^2_{L_2(\widetilde{\Pi}\e)}
+\suml_{j\in\mathcal{N}\e}\|u\e\|_{L_2(T_j\e)}^2+ 
\suml_{j\in\mathcal{N}\e}|B_j\e|\left|\< u\e \> _{B_j\e}\right|^2+
\suml_{j\in\mathcal{N}\e}\|u\e-\<u\e\>_{B_j\e}\|^2_{L_2(B_j\e)},
\end{gather*}
and hence, in view of \eqref{uT}-\eqref{uPi}, we obtain
\begin{gather}\label{1}
\suml_{j\in\mathcal{N}\e}|B_j\e| \left|\<u\e\>_{B_j\e}\right|^2=1+o(1)\text{ as }\eps\to 0.
\end{gather}

From \eqref{1}, taking into account \eqref{ass2}, \eqref{ass3}, we obtain the asymptotics for $v\e$:
\begin{gather}\label{v1}
\|\nabla v\e\|^2_{L_2(\Omega\e)}=
\suml_{j\in\mathcal{N}\e}{d\e}(h\e)^{-1}\left|\<u\e\>_{B_j\e}\right|^2=q  + o(1)\text{ as }\eps\to 0,\\\label{v2}
\suml_{j\in\mathcal{N}\e}\|v\e\|^2_{L_2(B_j\e)}= \suml_{j\in\mathcal{N}\e}|B_j\e|\left| \<u\e\>_{B_j\e}\right|^2 =1+o(1)\text{ as }\eps\to 0,\\\label{v3}
\suml_{j\in\mathcal{N}\e}\|v\e\|^2_{L_2(T_j\e)}={1\over 3|B|}{(h\e)^2 }{d\e\over h\e\eps^2}\suml_{j\in\mathcal{N}\e}|B_j\e| \left| \<u\e\>_{B_j\e}\right|^2
=o(1)\text{ as }\eps\to 0.
\end{gather}
Asymptotics \eqref{v2}-\eqref{v3} together with $v\e|_{\widetilde{\Pi}\e}=0$ yield
\begin{gather}\label{v4}
\|v\e\|_{L_2(\widetilde{\Omega}\e)}=1+o(1)\text{ as }\eps\to 0.
\end{gather}

By virtue of \eqref{uT}-\eqref{uPi} and \eqref{v3} 
\begin{gather}\label{u-v}
\|u\e-v\e\|^2_{L_2(\widetilde{\Omega}\e)}=
\suml_{j\in\mathcal{N}\e}\left(\|u\e-\<u\e\>_{B_j\e}\|^2 _{L_2(B_j\e)}+
\|u\e-v\e\|^2_{L_2(  T_j\e)}\right)+
\|u\e\|^2_{L_2(\widetilde{\Pi}\e)}
\underset{\eps\to 0}\to 0.
\end{gather}

Since $(u\e,u_k^{\varphi\e,\eps})_{L_2(\widetilde{\Omega}\e)}=0$ for $k=1,\dots, k\e-1$, we get, using the Bessel inequality:
\begin{gather*}
\|w\e\|^2_{L_2(\widetilde{\Omega}\e)}=\suml_{k=1}^{k\e-1}\left|(v\e,u_k^{\varphi\e,\eps})_{L_2(\widetilde{\Omega}\e)}\right|^2=
\suml_{k=1}^{k\e-1}\left|(v\e-u\e,u_k^{\varphi\e,\eps})_{L_2(\widetilde{\Omega}\e)}\right|^2\leq \|v\e-u\e\|^2_{L_2(\Omega\e)},\\
\|\nabla w\e\|^2_{L_2(\widetilde{\Omega}\e)}=\suml_{k=1}^{k\e-1}\lambda_k^{\varphi\e,\eps}\left|(v\e, u_k^{\varphi\e,\eps})_{L_2(\widetilde{\Omega}\e)}\right|^2=
\suml_{k=1}^{k\e-1}\lambda_k^{\varphi\e,\eps}\left|(v\e-u\e,u_k^{\varphi\e,\eps})_{L_2(\widetilde{\Omega}\e)}\right|^2\leq \lambda\e\|v\e-u\e\|^2_{L_2(\widetilde{\Omega}\e)},
\end{gather*}
whence, in view of \eqref{u-v},
\begin{gather}\label{g_est}
\|w\e\|^2_{H^1(\widetilde{\Omega}\e)}\to 0\text{ as }\eps\to 0.
\end{gather}

Now, we are in position to estimate the remainder $\delta\e$. 
One has the following variational characterization for $\lambda\e$ (see, e.g., \citep{RS72}):
\begin{gather}\label{minmax_p}
\lambda\e=\inf\left\{{\|\nabla u\|^2_{L_2(\widetilde{\Omega}\e)}\over \| u\|^2_{L_2(\widetilde{\Omega}\e)}},\ 0\not=u\in \mathrm{dom}(\mathfrak{h}^{\phi\e,\eps})\cap \left(\mathrm{span}\left\{u_1^{\varphi\e,\eps},\dots,u_{k\e-1}^{\varphi\e,\eps}\right\}\right)^\perp \right\}.
\end{gather}
From  \eqref{minmax_p} we get, taking into account  \eqref{g_prop1}:
\begin{gather}\label{star}
\lambda\e=\|\nabla u\e\|^2_{L_2(\widetilde{\Omega}\e)}\leq {\|\nabla \tilde v\e\|^2_{L_2(\widetilde{\Omega}\e)}\over \|\tilde v\e\|^2_{L_2(\widetilde{\Omega}\e)}},\text{ where }\tilde v\e=v\e-w\e.
\end{gather}
 Inequality \eqref{star} is equivalent to
\begin{gather}\label{minmax}
\|\nabla \delta\e\|^2_{L_2(\widetilde{\Omega}\e)}\leq {\|\nabla \tilde v\e\|^2_{L_2(\widetilde{\Omega}\e)}\left(\|\tilde v\e\|^{-2}_{L_2(\widetilde{\Omega}\e)}-1\right)-2(\nabla \tilde v\e,\nabla \delta\e)_{L_2(\widetilde{\Omega}\e)}}.
\end{gather}

Due to \eqref{v1}, \eqref{v4}, \eqref{g_est}
\begin{gather}\label{minmax1}
{\|\nabla \tilde v\e\|^2_{L_2(\widetilde{\Omega}\e)}}\left(\|\tilde v\e\|^{-2}_{L_2(\widetilde{\Omega}\e)}-1\right)\to 0\text{ as }\eps\to 0.
\end{gather}

Now, let us estimate the last term in the right-hand-side of \eqref{minmax}. One has
\begin{multline}\label{minmax2}
(\nabla \tilde v\e,\nabla \delta\e)_{L_2(\widetilde{\Omega}\e)}=(\nabla v\e,\nabla \delta\e)_{L_2(\widetilde{\Omega}\e)}-(\nabla w\e,\nabla \delta\e)_{L_2(\widetilde{\Omega}\e)}\\=(\nabla v\e,\nabla u\e-\nabla v\e)_{L_2(\widetilde{\Omega}\e)}+(\nabla v\e,\nabla w\e)_{L_2(\widetilde{\Omega}\e)}-(\nabla w\e,\nabla \delta\e)_{L_2(\widetilde{\Omega}\e)}.
\end{multline}

Integrating by parts and taking into account that $\Delta v_j\e=0$ in $T_j\e$ we get
\begin{gather*}
(\nabla v\e,\nabla u\e-\nabla v\e)_{L_2(\widetilde{\Omega}\e)}=\suml_{j\in \mathcal{N}\e}\intl_{T_j\e} \nabla v\e\cdot\nabla (u\e-v\e)\d \x
={d\e\over h\e}\suml_{j\in \mathcal{N}\e} \langle u\e \rangle_{B_j^{\eps}} \left(-\langle u\e \rangle_{S_j\e}+\langle u\e \rangle_{C_j\e}-\langle u\e \rangle_{B_j^{\eps}}\right).
\end{gather*}
Using   \eqref{ass1}, \eqref{ass3}, \eqref{lm2est1}, \eqref{lm2est2}, \eqref{uPi}, \eqref{1} we obtain
\begin{multline}\label{delta2}
\left|(\nabla v\e,\nabla u\e-\nabla v\e)_{L_2(\widetilde{\Omega}\e)}\right|^2\leq  
\left({d\e\over h\e}\right)^2 |B_j\e|^{-1}
\left\{\suml_{j\in \mathcal{N}\e}|B_j\e|\left|\langle u\e \rangle_{B_j^{\eps}}\right|^2 \right\}
\left\{\suml_{j\in\mathcal{N}\e}\left|\langle u\e \rangle_{S_j\e}+\langle u\e \rangle_{C_j\e}-\langle u\e \rangle_{B_j^{\eps}}\right|^2 \right\}\\\leq 
C_1 \left({d\e\over h\e \eps^2}\right)^2 \eps^2|B_j\e|^{-1}
\suml_{j\in\mathcal{N}\e}\left(\eps^2\left|\langle u\e \rangle_{Y_j^{\eps}}\right|^2 +\eps^2\left|\langle u\e \rangle_{S_j\e}-\langle u\e\rangle_{Y_j\e}\right|^2 +\eps^2\left|\langle  u\e \rangle_{C_j\e}-\langle u\e \rangle_{B_j^{\eps}}\right|^2 \right)\\\leq C_2\left(\|u\e\|_{L_2(\widetilde{\Pi}\e)}^2+ \eps^2|\ln d\e|\|\nabla u\e\|_{L_2(\cupl_{j\in\mathcal{N}\e}Y_j\e)}^2+ \eps^2|\ln d\e|\|\nabla u\e\|_{L_2(\cupl_{j\in\mathcal{N}\e} B_j\e)}^2\right)\to 0\text{ as }\eps\to 0.
\end{multline}
Also, by virtue of \eqref{v1}, \eqref{g_est}, 
\begin{gather}\label{minmax4}
(\nabla v\e,\nabla w\e)_{L_2(\widetilde{\Omega}\e)}\to 0\text{ as }\eps\to 0,
\\\label{minmax5}
\left|(\nabla w\e,\nabla \delta\e)_{L_2(\widetilde{\Omega}\e)}\right|\leq \left|(\nabla w\e,\nabla u\e)_{L_2(\widetilde{\Omega}\e)}\right|+\left|(\nabla w\e,\nabla v\e)_{L_2(\widetilde{\Omega}\e)}\right|+\|\nabla w\e\|^2_{L_2(\widetilde{\Omega}\e)}\to 0\text{ as }\eps\to 0.
\end{gather}
From \eqref{minmax2}-\eqref{minmax5} we get
\begin{gather}\label{delta3}
\liml_{\eps\to 0}(\nabla \tilde v\e,\nabla \delta\e)_{L_2(\widetilde{\Omega}\e)}=0.
\end{gather}

Finally, combining \eqref{minmax}, \eqref{minmax1} and \eqref{delta3} we conclude that
\begin{gather}\label{delta4}
\liml_{\eps\to 0}\|\nabla \delta\e\|_{L_2(\widetilde{\Omega}\e)}=0,
\end{gather}
whence, taking into account \eqref{v1},  \eqref{g_est}, 
\begin{gather*}
\lambda\e = \|\nabla u\e\|^2_{L_2(\widetilde{\Omega}\e)}\sim  \|\nabla v\e\|^2_{L_2(\widetilde{\Omega}\e)}\sim q \text{ as }\eps\to 0.\text{ Q.E.D.}
\end{gather*}

Property (i) is completely proved.

\subsection{Proof of the property (ii)}

Let $\lambda\in \sigma(\mathcal{H})$; we have to show that there exists a family $\left\{\lambda^{\eps}\in\sigma(\mathcal{H}\e)\right\}_{\eps}$ such that   $\liml_{\eps\to 0}\lambda\e=\lambda$.

Let us assume the opposite. Then there exist a
subsequence $\eps_k$, $\eps_k\searrow 0$ and
 $\delta>0$ such that
\begin{gather}
\label{dist} (\lambda-\delta,\lambda+\delta)\cap\sigma(\mathcal{H}^{\eps})=\varnothing\text{ as }\eps=\eps_k.
\end{gather}

Since $\lambda\in\sigma(\mathcal{H})$ there exists
$F=(f_1,f_2) \in [L_2(\mathbb{R})]^2$, such that  
\begin{gather}\label{notinim}
F\notin\mathrm{range}(\mathcal{H} -\lambda {I}).
\end{gather}

Due to \eqref{dist} $\lambda$ is not in the spectrum of $\mathcal{H}\e$ as $\eps=\eps_k$ and therefore for an arbitrary $f\e\in L_2(\Omega\e)$  there exists the unique solution $u\e$ of the problem
\begin{gather}\label{problem}
\mathcal{H}\e u\e-\lambda u\e=f\e,\ \eps=\eps_k.
\end{gather}
Moreover the following estimates hold true:
\begin{gather}\label{uL2+uH1+}
\|u\e\|_{L_2(\Omega\e)}\leq {1\over \delta}\|f\e\|_{L_2(\Omega\e)},\quad 
\|\nabla u\e\|_{L_2(\Omega\e)}^2= \lambda\|u\e\|^2_{L_2(\Omega\e)}+(f\e,u\e)_{L_2(\Omega\e)}\leq \left({\lambda\over\delta^2}+{1\over \delta}\right)\|f\e\|_{L_2(\Omega\e)}^2.
\end{gather}

Now, we choose   $f\e $ in \eqref{problem} by \eqref{J*}. The family 
$\{f\e\}_\eps$ satisfies \eqref{f}, whence, taking into account \eqref{uL2+uH1+}, we conclude that there exist a subsequence (still indexed by $\eps_k$) and $u_1\in H^1(\mathbb{R})$, $u_2\in L_2(\mathbb{R})$ such that
\eqref{conv12} hold (as $\eps=\eps_k\to 0$).

Repeating word-by-word the proof of Theorem \ref{th1} we conclude that
$U=(u_1,u_2)$ solves \eqref{notinim}. We obtain a contradiction. 
Property (ii) is   proved and this finishes the proof of Theorem \ref{th2}.

\section{Proof of Lemma \ref{lm1} \label{sec6}}

In the proof we   deal with domains $B_0\e,  T_0\e, Y_0\e$.
For convenience hereinafter  we omit the  index ``$0$''.
We denote
$$G\e=Y\e\cup T\e,\quad D\e=B\e\cup G\e.$$
The set $D\e$ is the smallest period cell for the operators $\mathcal{H}\e$. 
Let $\varphi\in \mathbb{R}\backslash(2\pi\mathbb{Z})$. In $L_2(D\e)$ we introduce the sesquilinear form ${\mathfrak{a}}^{\varphi,\eps}$ by
\begin{gather}\label{etaform}\hspace{-5mm}
\mathfrak{a}^{\varphi,\eps}[u,v]=\intl_{D\e}\nabla u\cdot\overline{\nabla v} \d \x,\ 
\mathrm{dom}(\mathfrak{a}^{\varphi,\eps})=\left\{u\in H^1(D\e): u(\eps,x_2)=\exp(i\varphi)u(0,x_2)\text{ for }x_2\in (-\eps,0)\right\}.
\end{gather}
We denote by $\mathcal{A}^{\varphi,\eps}$ the operator  associated with this form, by $\left\{\mu_{k}^{\varphi,\eps}\right\}_{k\in\mathbb{N}}$ we denote the sequence of its eigenvalues  arranged in the ascending order and repeated according to  their multiplicity. 

Again using Floquet-Bloch theory we get the representation
\begin{gather}\label{Floquet+}
\sigma(\mathcal{H}\e)=\cupl_{k\in\mathbb{N}}\left\{\mu_{k}^{\varphi,\eps}:\ \varphi\in \mathbb{R}\backslash(2\pi\mathbb{Z})\right\}.
\end{gather}
The sets $\left\{\mu_{k}^{\varphi,\eps}:\ \varphi\in \mathbb{R}\backslash(2\pi\mathbb{Z})\right\}$ are compact intervals. 
Our goal is to prove that
\begin{gather}\label{infty}
\mu_{2}^{\varphi,\eps}\to \infty\text{ as }\eps\to 0\text{ provided }\varphi\not=0,
\end{gather}
then  Lemma \ref{lm1} follows from directly from \eqref{Floquet+}-\eqref{infty}.

To prove \eqref{infty} we consider the following operator in $L_2(D\e)=L_2(B\e)\oplus L_2(G\e)$:
$$\widehat{\mathcal{A}}^{\varphi,\eps}= \left(-\Delta_{B\e}\right)\oplus \left(-\Delta^{\varphi}_{G\e}\right).$$
Here $-\Delta_{B\e} $ is the Neumann Laplacian on $B\e$, 
and  $-\Delta^{\varphi}_{G\e}$ is the operator acting in $L_2(G\e)$ 
being associated with the sesqulilinear form, which is defined by \eqref{etaform} with $G\e$ instead of $D\e$.
We denote  by  $\left\{\widehat{\mu}_{k}^{\varphi,\eps}\right\}_{k\in\mathbb{N}}$  the sequence of eigenvalues of $\widehat{\mathcal{A}}^{\varphi,\eps}$.
It is easy to see that $\mathrm{dom}(\widehat{\mathcal{A}}^{\varphi,\eps})\supset \mathrm{dom}({\mathcal{A}}^{\varphi,\eps})$ and $\widehat{\mathcal{A}}^{\varphi,\eps}={\mathcal{A}}^{\varphi,\eps}$ on $\mathrm{dom}({\mathcal{A}}^{\varphi,\eps})$.
Then, by the min-max principle,
\begin{gather}\label{minmax_lm}
\forall k\in\mathbb{N}:\quad\widehat{\mu}_{k}^{\varphi,\eps}\leq {\mu}_{k}^{\varphi,\eps}.
\end{gather}

The first eigenvalue of $-\Delta _{B\e}$ is equal to zero, therefore $\widehat{\mu}_{1}^{\varphi,\eps}=0$.
Let us prove that the first eigenvalue of $-\Delta^{\varphi}_{G\e}$ tends to 
infinity. For an arbitrary $u\in H^1(Y\e)$ one has, using the Poincar\'{e} inequality:
\begin{gather}\label{estim1+}
\|u\|^2_{L_2(Y\e)}=\|u-\< u\>_{Y\e}\|^2_{L_2(Y\e)}+\eps^2 |\< u\>_{Y\e}|^2
\leq C\eps^2\|\nabla u\|^2_{L_2(Y\e)}+\eps^2 |\< u\>_{Y\e}|^2.
\end{gather}

We denote:
$$Z_0\e= \{\x\in\mathbb{R}^2:\ x_1=0,\ x_2\in (-\eps,0)\},\quad 
Z_1\e= \{\x\in\mathbb{R}^2:\ x_1=\eps,\ x_2\in (-\eps,0)\}.$$
Employing  the trace  and the Poincar\'{e} inequalities  one has
\begin{multline}\label{estim1++}
k=0,1:\quad \left|\< u \>_{Z_k\e} -  \< u\>_{Y\e}\right|^2=
\left|\left\< u - \< u\>_{Y\e}\right\>_{Z_k\e}\right|^2\leq
\eps^{-1}\left\|u  - \< u\>_{Y\e}\right\|^2_{L_2(Z_k\e)}\\\leq 
C\left(\eps^{-2}\left\|u - \< u \>_{Y\e}\right\|^2_{L_2(Y\e)}+  \left\|\nabla u  \right\|^2_{L_2(Y\e)}\right)\leq C_1 \left\|\nabla u  \right\|^2_{L_2(Y\e)}.
\end{multline}

Now, suppose that $u$ is not only in $H^1(G\e)$, but also
 $u(\eps,x_2)=\exp(i\varphi)u(x_2)\text{ for }x_2\in (-\eps,0)$.
Then it follows from \eqref{estim1++} that (recall: $\varphi\not=0$, whence $1-\exp(i\varphi)\not=0$)
\begin{multline}\label{estim2+}\left|\< u \>_{Y\e}\right|^2=|1-\exp(i\varphi)|^{-2}
\left|\< u\>_{Y\e}-\< u\>_{Z_1\e}+
\exp(i\varphi)\< u\>_{Z\e_0}- \exp(i\varphi)\< u\>_{Y\e} \right|^2\\\leq
2|1-\exp(i\varphi)|^{-2}\left(
\left|\< u\>_{Y\e}-\< u\>_{Z_1\e}\right|^2+
\left|\< u\>_{Z\e_0}-\< u\>_{Y\e}\right|^2
\right)\leq C\|\nabla u\|^2_{L_2(Y\e)}.
\end{multline}
Combining \eqref{estim1+} and \eqref{estim2+} we arrive at
\begin{gather}\label{estim3+}
\|u\|^2_{L_2(Y\e)}\leq C\eps^2 \|\nabla u\|^2_{L_2(Y\e)}.
\end{gather}
Then, using Lemma \ref{lm4} and inequality \eqref{estim3+}, we obtain the estimate
\begin{gather}\label{estim4+}
\|u\|^2_{L_2(T\e)}\leq 
 C  \left( (h\e \eps)^2 \|\nabla u\|^2_{L_2(Y\e)} +d\e|\ln d\e|h\e  \|\nabla u\|^2_{L_2(Y\e)}+  (h\e)^2\|\nabla u\|^2_{L_2 (  T_j\e )}\right).
\end{gather}

It follows from \eqref{estim3+}-\eqref{estim4+} that  for each $u\in \left\{v\in H^1(G\e):\  v(\eps,x_2)=\exp(i\varphi)v(x_2)\text{ for }x_2\in (-\eps,0)\right\}$
\begin{gather}\label{estim5+}
\|u\|^2_{L_2(G\e)}\leq 
 C\eta\e \|\nabla u\|^2_{L_2(G\e)},\text{ where }\eta\e\to 0\text{ as }\eps\to 0.
\end{gather}
Inequality \eqref{estim5+} implies that the first eigenvalue of  the operator $-\Delta^{\varphi}_{L_2(G\e)}$ tends to infinity as $\eps\to 0$. Evidently, the second eigenvalue of $-\Delta^{\varphi}_{B\e}$ also tends to infinity.
Therefore
\begin{gather*}
\widehat{\mu}_{2}^{\varphi,\eps}\to \infty\text{ as }\eps\to 0\text{ provided }\varphi\not=0,
\end{gather*}
whence, using \eqref{minmax_lm}, we infer \eqref{infty}. 
Lemma \ref{lm1} is proved.

\section*{Acknowledgements}

A.K. gratefully acknowledges financial support by the Deutsche
Forschungsgemeinschaft (DFG) through CRC 1173 ``Wave phenomena: analysis and numerics''  
and the project ``Analisi asintotica di problemi spettrali in domini perturbati'' GNAMPA (Gruppo Nazionale per l'Analisi Matematica, la Probabilit\`a e le loro Applicazioni) of INDAM (Istituto Nazionale Di Alta Matematica). 
G.C. is a member of GNAMPA (INDAM).


\begin{thebibliography}{99}



\bibitem{AHH} J.M.~Arrieta, J.K.~Hale, Q.~Han, Eigenvalue problems for non-smoothly perturbed domains, J. Differ. Equations 91 (1991) 24--52.

\bibitem{AP11} J.M.~Arrieta, M.C.~Pereira,  Homogenization in a thin domain with an oscillatory boundary, J. Math. Pures Appl. 96 (2011) 29--57.

\bibitem{AP13} J.M.~Arrieta, M.C.~Pereira, The Neumann problem in thin domains with very highly oscillatory boundaries, J. Math. Anal. Appl. 404 (2013) 86--104. 

\bibitem{BNR13} F.L.~Bakharev, S.A.~Nazarov, K.M.~Ruotsalainen, A gap in the spectrum of the Neumann-Laplacian on a periodic waveguide, Appl. Analysis 92 (2013) 1889--1915.

\bibitem{BN15}
F.L.~Bakharev, S.A.~Nazarov,  
Gaps in the spectrum of a waveguide composed of domains with different limiting dimensions, Sib. Math. J. 56 (2015) 575--592 (2015).

\bibitem{BRT15}
F.L.~Bakharev, K.~Ruotsalainen, J.~Taskinen, 
Spectral gaps for the linear surface wave model in periodic channels,
Quart. J. Mech. Appl. Math. 67 (2014) 343--362. 

\bibitem{BP13_1} D.I.~Borisov, K.V.~Pankrashkin, Gap opening and split band edges in waveguides coupled by a periodic system of small windows, Math. Notes 93 (2013) 660--675.

\bibitem{BP13_2} D.~Borisov, K.~Pankrashkin, Quantum waveguides with small periodic perturbations: gaps and edges of Brillouin zones, J. Phys. A: Math. Theor. 46 (2013) 235203.

\bibitem{B15} D.I.~Borisov,  
On the band spectrum of a Schrödinger operator in a periodic system of domains coupled by small windows,
Russ. J. Math. Phys. 22 (2015) 153--160. 


\bibitem{CH53} R.~Courant, D.~Hilbert, Methods of Mathematical Physics, Wiley-Interscience, New York, 1953.

\bibitem{CK15}
G.~Cardone, A.~Khrabustovskyi,
Neumann spectral problem in a domain with very corrugated boundary,
J. Differ. Equations 259 (2015) 2333--2367.

\bibitem{CK16}
G.~Cardone, A.~Khrabustovskyi,
Example of periodic Neumann waveguide with gap in spectrum.
to appear in: J. Dittrich, H. Kova\v{r}ik, A. Laptev (Eds.), Functional Analysis and Operator Theory for Quantum Physics. A Festschrift in Honor of Pavel Exner, EMS Publ. House, 2016;  	arXiv:1605.07812.

\bibitem{CMN09} G.~Cardone, V.~Minutolo, S.~Nazarov, Gaps in the essential spectrum of periodic elastic waveguides, Z. Angew. Math. Mech. 89 (2009) 729--741.

\bibitem{CNP10} G.~Cardone, S.~Nazarov, C.~Perugia, A gap in the essential spectrum of a cylindrical waveguide with a periodic perturbation of the surface, Math. Nachr. 283 (2010) 1222--1244.


\bibitem{DLPSW11} W.~D\"{o}rfler, A.~Lechleiter, M.~Plum, G.~Schneider,
C.~Wieners, Photonic Crystals. Mathematical Analysis and Numerical
Approximation, Springer, Berlin, 2011.
 
\bibitem{FK96} A.~Figotin, P.~Kuchment, Band-gap structure of the spectrum
of periodic dielectric and acoustic media. I. Scalar model, SIAM
J. Appl. Math. 56 (1996), 68--88; II. Two-dimensional
photonic crystals, ibid., 1561--1620. 


\bibitem{FS08} L.~Friedlander, M.~Solomyak, On the spectrum of narrow periodic waveguides, Russian J. Math. Phys. 15 (2008) 238--242.

\bibitem{HR92}
J.K.~Hale, G.~Raugel, Reaction-diffusion equation on thin domains, J. Math. Pures Appl. 71 (1992) 33--95.

\bibitem{KP16}
A.~Khrabustovskyi, M.~Plum, Spectral properties of elliptic operators with double-contrast coefficients near a hyperplane, Asymptot. Analysis 98 (2016), 91--130.

\bibitem{K93} P.~Kuchment, Floquet Theory For
Partial Differential Equations, Birkhauser, Basel, 1993.

\bibitem{K01} P.~Kuchment, The mathematics of photonic crystals. In: {Mathematical Modeling in Optical Science}, Frontiers in Applied
Mathematics 22, Philadelphia, PA: SIAM, 2001, Chap. 7, pp.
207--272.

\bibitem{MP10} T.A.~Mel'nyk, A.V.~Popov, 
Asymptotic analysis of boundary value problems in thin perforated domains with rapidly changing thickness, Nonlinear Oscil. 13 (2010) 57--84.

\bibitem{MP12} T.A.~Mel'nyk, A.V.~Popov,  Asymptotic analysis of boundary value and spectral problems in thin perforated domains with rapidly changing thickness and different limiting dimensions,  Sb. Math. 203 (2012) 1169--1195

\bibitem{NRT10}
S.A.~Nazarov, K.~Ruotsalainen, J.~Taskinen, Essential spectrum of a periodic elastic waveguide may contain arbitrarily many gaps, Appl. Anal. 89 (2010)  109--124.

\bibitem{N10_1} S.A.~Nazarov, Opening of a gap in the continuous spectrum of a periodically perturbed waveguide, Math. Notes 87 (2010) 738--756.

\bibitem{N10_2} S.A.~Nazarov, An example of multiple gaps in the spectrum of a periodic waveguide,
Sb. Math. 201(4) (2010) 569--594.


\bibitem{N12} S.A.~Nazarov, The asymptotic analysis of gaps in the spectrum of a waveguide perturbed with a periodic family of small voids, J. Math. Sci. 186 (2012) 247--301.

\bibitem{N15} S.A.~Nazarov, 
Gap opening around a given point of the spectrum of a cylindrical waveguide by means of a gentle periodic perturbation of walls,
J. Math. Sci. 206 (2015)  288--314.


\bibitem{P10} K.~Pankrashkin, On the spectrum of a waveguide with periodic cracks, J. Phys. A: Math. Theor.  43
(2010), article ID: 474030.

\bibitem{RS72} M.~Reed, B.~Simon, Methods of Modern Mathematical Physics. I. 
Funtional analysis, Academic Press, New York - San Francisco - London, 1972; IV. Analysis of Operators, ibid, 1979.

\bibitem{RT75} J.~Rauch, M.~Taylor, 
Potential and scattering theory on wildly perturbed domains,
J. Funct. Anal. 18 (1975) 27--59.

\bibitem{PS13}
M.C.~Pereira, R.P.~Silva, 
Error estimates for a Neumann problem in highly oscillating thin domains,
Discrete Contin. Dyn. Syst. 33 (2013) 803--817.

\bibitem{PS15}
M.C.~Pereira, R.P.~Silva,  Correctors for the Neumann problem in thin domains with locally periodic oscillatory structure, Quart. Appl. Math. 73 (2015) 537--552. 

\bibitem{SP80} E.~S\'{a}nchez-Palencia, Nonhomogeneous Media and Vibration 
Theory, Springer-Verlag, Berlin - New York, 1980.

\bibitem{Y98} K.~Yoshitomi, Band gap of the spectrum in periodically curved quantum waveguides,
J. Differ. Equations 142 (1998) 123--166.

\end{thebibliography}
\end{document}